\documentclass[11pt, leqno]{amsart}
\usepackage{amsfonts,amssymb, amscd}
\usepackage{amsmath}
\usepackage{lmodern}
\usepackage{mathrsfs}
\usepackage{amsthm}
\usepackage{qsymbols}
\usepackage{mathtools}
\mathtoolsset{showonlyrefs}
\usepackage{dsfont}
\usepackage{graphicx}
\usepackage{latexsym}
\usepackage{chngcntr}
\usepackage[noadjust]{cite}
\usepackage{paralist}
\usepackage[parfill]{parskip}
\usepackage{bm}
\usepackage{tikz-cd}
\usepackage{nicefrac}
\usepackage{float}
\usepackage{csquotes}
\usepackage{todonotes}
\makeatletter
\@mparswitchfalse%
\makeatother
\normalmarginpar 

\usepackage[shortlabels]{enumitem}
\setlist[enumerate,1]{label = \normalfont(\roman*), ref = (\roman*)}

\newtheorem{theorem}{Theorem}[section]
\newtheorem{lemma}[theorem]{Lemma}
\newtheorem{proposition}[theorem]{Proposition}
\newtheorem{corollary}[theorem]{Corollary}

\theoremstyle{definition}
\newtheorem{definition}[theorem]{Definition}
\newtheorem{remark}[theorem]{Remark}

\newtheorem{ass}[theorem]{Assumption}
\newtheorem{problem}[theorem]{Problem}
\newtheorem{conv}[theorem]{Convention}

\usepackage[plainpages=false,pdfpagelabels,
citecolor=red,backref=page]{hyperref}
\usepackage{xcolor}
\hypersetup{
	colorlinks,
	linkcolor={cyan!90!black},
	citecolor={magenta},
	urlcolor={green!40!black}
} %
\usepackage{doi}

\usepackage[left,final]{showlabels}

\newcommand{\R}{\mathbb{R}}
\newcommand{\C}{\mathbb{C}}
\newcommand{\N}{\mathbb{N}}
\renewcommand{\L}{\mathrm{L}}
\newcommand{\W}{\mathrm{W}}
\newcommand{\Cont}{\mathrm{C}}
\newcommand{\Lip}{\mathrm{Lip}}
\newcommand{\E}{\mathrm{E}_0}
\newcommand{\Egam}{{\mathbb{E}_\frac{1}{2}}}
\newcommand{\V}{\mathrm{E}_1}
\newcommand{\Ep}{\mathrm{E}^p_0}
\newcommand{\Epgam}{{\mathbb{E}^{p}_\frac{1}{2}}}
\newcommand{\Vp}{\mathrm{E}^p_1}
\newcommand{\Eq}{\mathrm{E}^q_0}
\newcommand{\Eqgam}{{\mathbb{E}^{q}_\frac{1}{2}}}
\newcommand{\Vq}{\mathrm{E}^q_1}
\newcommand{\cL}{\mathcal{L}}
\newcommand{\cF}{\mathcal{F}}
\newcommand{\e}{\mathrm{e}}
\renewcommand{\d}{\,\mathrm{d}}

\newcommand{\eps}{\varepsilon}
\newcommand{\B}{\mathrm{B}}

\renewcommand{\H}{\mathrm{H}}

\renewcommand\Re{\operatorname{Re}}

\renewcommand{\div}{\operatorname{div}}

\renewcommand{\i}{\mathrm{i}}

\newcommand{\Exp}{\mathbb{E}}
\newcommand{\HS}{\mathcal{L}_2}

\newcommand{\Dom}{D}
\newcommand{\wt}{\widetilde}
\newcommand{\X}{\mathcal{X}}
\newcommand{\BMSpace}{\mathcal{U}}
\newcommand{\Law}{\mathcal{L}}
\renewcommand{\P}{\mathbb{P}}
\newcommand{\weak}{\mathrm{w}}
\newcommand{\I}{\textrm{I}}
\newcommand{\II}{\textrm{II}}
\newcommand{\III}{\textrm{III}}
\newcommand{\IV}{\textrm{IV}}
\newcommand{\VV}{\textrm{V}}
\newcommand{\VI}{\textrm{VI}}

\newcommand{\lb}{\langle}
\newcommand{\rb}{\rangle}

\newcommand{\wh}{\widehat}

\newcommand{\calL}{\mathcal{L}}

\makeatletter
\g@addto@macro\bfseries{\boldmath}
\makeatother

\makeatletter
\@namedef{subjclassname@2020}{%
	\textup{2020} Mathematics Subject Classification}
\makeatother

\title[An extrapolation result in the variational setting]{An extrapolation result in the variational setting: improved regularity, compactness, and applications to quasilinear systems}
\author{Sebastian Bechtel}
\address{Delft Institute of Applied Mathematics, Delft University of Technology, P.O. Box 5031, 2600 GA Delft, The Netherlands}
\email{S.Bechtel@tudelft.nl}
\author{Mark Veraar}
\email{M.C.Veraar@tudelft.nl}
\thanks{The first author is supported by the Alexander von Humboldt foundation by a Feodor Lynen grant. The second author is supported by the VICI subsidy VI.C.212.027 of the Netherlands Organisation for Scientific Research (NWO)}
\date{\today}
\dedicatory{}

\allowdisplaybreaks

\begin{document}
	\begin{abstract}
In this paper we consider the variational setting for SPDE on a Gelfand triple $(V, H, V^*)$. Under the standard conditions on a linear coercive pair $(A,B)$, and a symmetry condition on $A$ we manage to extrapolate the classical $\L^2$-estimates in time to $\L^p$-estimates for some $p>2$ without any further conditions on $(A,B)$. As a consequence we obtain several other a priori regularity results of the paths of the solution.

Under the assumption that $V$ embeds compactly into $H$, we derive a universal compactness result quantifying over all $(A,B)$.
As an application of the compactness result we prove global existence of weak solutions to a system of second order quasi-linear equations.
    \end{abstract}
	\maketitle

	\section{Introduction}
	\label{sec:intro}

In this paper we consider stochastic evolution equations in the variational setting. This is the stochastic version of Lions classical setting \cite{Lio69} and goes back to the work of \cite{bensoussan_equations_1972, KR79, Pardoux}. For details we refer the reader to the monographs \cite{LR,Rozov}.

This paper consists of two main parts. In the first part we prove an extrapolation result, which
improves the usual a priori regularity estimates for linear equations with operators that satisfy the usual coercivity conditions. As a consequence we derive a new type of compactness result for linear equations, which is universal in the sense that for given constants in the estimates it varies over all admissible operators and data.
In the second part we apply the compactness result to obtain global existence for a class of stochastic parabolic systems which is not well-understood yet.

The triple $(V, H, V^*)$ are Hilbert spaces such that $V\hookrightarrow H\hookrightarrow V^*$ densely, where we identify the Hilbert space $H$
with its dual, and where $V^*$ is the dual with respect to the inner product of $H$. The stochastic equations we consider can be written as:
\begin{equation}
\label{eq:SEEintro}
\tag{LP}
\left\{
\begin{aligned}
        \d u(t) + A(t)u(t) \, \d t  & = f(t)\, \d t + (B(t)u(t) + g(t))\, \d W(t),
\\ u(0) &= u_0.
\end{aligned}
\right.
\end{equation}
Here $A:(0,T)\times\Omega\to \calL(V, V^*)$ and $B:(0,T)\times\Omega\to \calL(V, \calL_2(U, H))$, and $W$ is a $U$-cylindrical Brownian motion, where $U$ is a real separable Hilbert space. The time $T>0$ is finite. For the inhomogeneities $(f,g,u_0)$ we assume $f:(0,T)\times\Omega\to V^*$, $g:(0,T)\times\Omega\to \calL_2(U,H)$ and $u_0:\Omega\to H$, and the standard measurability conditions are supposed to be satisfied.

Consider the following condition on the pair $(A,B)$.
\begin{ass}\label{ass:ABintro}
There exist constants $\Lambda, \lambda > 0$ and $M \geq 0$ such that pointwise in $[0,T]\times\Omega$ for all $v\in V$,
		\begin{align*}
			\Re\, \lb A v, v \rb - \frac{1}{2} \| Bv \|_{\HS(U,H)}^2 &\geq \lambda \| v \|_V^2 - M \| v \|_H^2 & \text{(coercivity),}
\\			\| Av \|_{V^*} \leq \Lambda \| v \|_V \quad & \& \quad \| Bv \|_{\HS(U,H)} \leq \Lambda \| v \|_V  & \text{(boundedness)}.
		\end{align*}
\end{ass}

Under Assumption \ref{ass:ABintro} it is standard that \eqref{eq:SEEintro} has a unique strong solution $u\in \L^2(0,T;V)\cap \Cont([0,T];H)$ a.s.\ and the following a priori estimate holds for a constant $C$ only depending on $T\vee 1$, $\lambda$, $\Lambda$, and $M$:
\begin{align}\label{eq:classicalvarintro}
\begin{aligned}
\|u\|_{\L^2(\Omega\times(0,T);V)}+  \|u\|_{\L^2(\Omega;\Cont([0,T];H))} \leq   C \big[\|u_0\|_{\L^2(\Omega;H)} &+ \|f\|_{\L^2(\Omega\times(0,T);V^*)} \\ & +\|g\|_{\L^2(\Omega\times(0,T);\calL_2(U,H))}\big].
\end{aligned}
\end{align}

In the deterministic setting it follows from \cite{DiElRe} that one can find some $p_0>2$ depending on $T\vee 1$, $\lambda$, $\Lambda$, and $M$ such that for every $p\in [2, p_0]$ there exists a constant $C_p$ such that
\begin{align}\label{eq:Lpdet}
\|u\|_{\L^p(0,T;V)} + \|u'\|_{\L^p(0,T;V^*)}\leq  C_p \big[\|u_0\|_{(H,V)_{1-\nicefrac{2}{p}}} &+ \|f\|_{\L^p(0,T;V^*)}\big].
\end{align}
Actually this even holds true outside the setting of Hilbert space and variational problems. On the other hand, \eqref{eq:Lpdet} does not hold for all $p\in (2, \infty)$.
In \cite{DiElRe} the estimate \eqref{eq:Lpdet} is used to obtain a compactness result, which in turn is used to prove global existence of solutions to a quasi-linear parabolic PDE.
On the other hand, if $A$ does not depend on time, then \eqref{eq:Lpdet} holds for all $p\in (1, \infty)$ (see \cite[Theorem 17.2.26~(4)]{HNVW3}). A stochastic version of an extrapolation result holds if $B = 0$ (or small), and was proved in \cite{LoristVeraar}. In the general case $B\neq 0$ the result is false since the $\L^p(\Omega)$-integrability can fail to hold (see \cite{BrzVer}) and a more restrictive coercivity condition is needed (see \cite{GHV}). However, it would be interesting to obtain some result with mixed integrability $\L^2(\Omega;\L^p(0,T;V))$ in case of general $(A,B)$ satisfying Assumption \ref{ass:ABintro} and any $p\in [2, \infty)$.

To prove \eqref{eq:Lpdet} one can use Sneiberg's lemma for complex interpolation. In the latter argument one needs an operator $T\in \calL(X_i,Y_i)$, where $(X_0, X_1)$ and $(Y_0, Y_1)$ are Banach couples. By complex interpolation also $T\in \calL(X_{\theta},Y_{\theta})$ for all $\theta\in (0,1)$. Sneiberg's result states that the interval of $\theta\in (0,1)$ for which $T$ is invertible is open. In the deterministic setting one can take $X_i = {}_{0}\W^{1,p_i}(0,T;V^*)\cap \L^{p_i}(0,T;V)$ and $Y_i = \L^{p_i}(0,T;V^*)$ with $1<p_0<2<p_1<\infty$ and $T u = u'+Au$, where ${}_{0}\W^{1,p_i}(0,T;V^*)$ is defined as the subspace of $\W^{1,p_i}(0,T;V^*)$ of function that vanish at zero. The boundedness of $T$ is trivial. Now if one takes $\theta\in (0,1)$ such that $\frac{1-\theta}{p_0} + \frac{\theta}{p_1} = \frac12$, then $T:X_{\theta}\to Y_{\theta}$ is invertible by \eqref{eq:Lpdet} for $p=2$ and with $u_0 = 0$. Thus, Sneiberg's lemma provides us with an open interval around this value of $\theta$, which yields the required result. As far as we could see it seems impossible to define an operator $T$ in the stochastic framework
so that the above steps can be completed.

\subsection*{New estimates for the linear problem}
In this paper we circumvent the above problem at the price of a symmetry condition on $A$, and prove the following result, where we emphasise that we do not need any regularity conditions on $(A,B)$.
\begin{theorem}[Extrapolation of integrability and regularity]\label{thm:introExtra}
Suppose that Assumption \ref{ass:ABintro} holds and that on $[0,T]\times \Omega$ and for $u,v\in V$ the
symmetry condition $\lb Au, v\rb = \lb Av, u\rb$ holds.
Then there exists $p_0>2$ only depending on $\Lambda, \lambda$  such that for all $p\in [2, p_0]$, $f\in \L^p_\cF(\Omega \times (0,T); V^*)$, $g\in \L^p_\cF(\Omega \times (0,T); \HS(U,H))$ and $u_0 \in \L^p_{\cF_0}(\Omega; (H,V)_{1-\nicefrac{2}{p},p})$, problem~\eqref{eq:SEEintro}
admits a unique solution $u \in \L^p(\Omega \times (0,T); V)$. Moreover, for $\delta\in (1/p,1/2)$ and $\theta\in [0,1/2)$ we have
\[u\in \L^p(\Omega; \Cont^{\delta-\nicefrac{1}{p}}([0,T]; [H, V]_{1-2\delta})),  \ \ \text{and} \ \ u\in \L^p(\Omega; \Cont([0,T]; (H, V)_{1-\nicefrac{2}{p},p})).\]
\end{theorem}
In the above we used the notation $[\cdot, \cdot]_{\theta}$ and $(\cdot, \cdot)_{\theta,p}$ for complex and real interpolation, respectively. The main novelties are $\L^p$-integrability in $\Omega\times (0,T)$ and the improved smoothness of the paths, which is much better than the classical path space $\Cont([0,T];H)$ used in \eqref{eq:classicalvarintro}.

Theorem \ref{thm:introExtra} is a special case of Theorem \ref{thm:main1}, where additional linear estimates are stated as well. As mentioned, we cannot use Sneiberg's lemma
to prove Theorem \ref{thm:introExtra} . Instead we will use abstract Stein interpolation. The proof was inspired by the recent work \cite{BoEg} on $\L^p$-bounds for heat semigroups.

The symmetry condition can be written equivalently as $A^* = A$. We do not know if the symmetry condition is needed in Theorem \ref{thm:introExtra}. It is certainly not necessary. For many operators $A$ satisfying Assumption \ref{ass:ABintro}, one can always apply Theorem \ref{thm:introExtra} to its symmetric part $\frac{A+A^*}{2}$ and use a perturbation argument $A = \frac{A+A^*}{2} + R$ with $R = \frac{A-A^*}{2}$ (see \cite[Theorem 3.2]{AV21_SMR_torus}). The only thing which needs to be checked is the following relative boundedness condition: for every $\eps>0$ there exists a $C_{\varepsilon}$ such that
\[\|Rv\|_{V^*}\leq \varepsilon \|v\|_V + C_{\varepsilon}\|v\|_H.\]
For instance for $2m$-th elliptic differential operators with smooth coefficients in space the highest order differentiation cancels out and $R$ turns out to be of order $2m-1$, and thus it satisfies the above relative boundedness condition.

Theorem \ref{thm:introExtra} implies that the $\L^2$-theory in \cite{AVvar} can be extended to $\L^p$-theory for some range $p\in (2, p_0]$. Although this leads to minimal changes in the proof, it has many consequences for applications since for several problems it turns out that any $p>2$ is sufficient. For instance this occurs for SPDEs in dimension $d=2$. Due to the fact that taking for instance $H = \H^1$,  $V = \H^2$ and $V^* = \L^2$, the space $H$ does not embed into $\L^\infty$ as this instance of the Sobolev embedding fails. Using $\L^p$-theory instead one only needs that $[H,V]_{1-\nicefrac{2}{p}} = \H^{2-\nicefrac1p}$ embeds into $\L^\infty$, which is true in dimension $d=2$. We will investigate these applications in a subsequent paper.

\subsection*{Universal compactness}

Thanks to Theorem \ref{thm:introExtra} we obtain a compactness/tightness result for the solution mapping $(f,g,u_0)\mapsto u$ for \eqref{eq:SEEintro} for the path space $\Cont([0,T];H)$. We expect that this has many consequences. In earlier works tightness of laws is obtained with $H$ replaced by $V^*$ or even larger spaces for different but related settings (cf.~\cite{DHV} and \cite{rockner2022well}).  Tightness plays a key role in the stochastic compactness method (see \cite{Hofmanovaetal}) for obtaining weak solutions. Using a smaller path space can give more information on the weak solution and its approximation.

\begin{theorem}[Universal compactness for variational problems]\label{thm:varcomp}
Suppose that the embedding $V\hookrightarrow H$ is compact. Let $\lambda,\Lambda,T>0$, $M\geq 0$ and $K> 0$ be fixed. Then there exists a $p_0>2$ only depending on $\lambda,\Lambda,T\vee 1$ such that for all $p\in (2, p_0]$, the laws $\{\mathscr{L}(u):u\}$ on $C([0,T];H)$ are tight, where $u$ runs over all strong solutions to \eqref{eq:SEEintro} and all $(A,B)$ satisfying Assumption \ref{ass:ABintro} with $A = A^*$, and all $f\in \L^p_\cF(\Omega \times (0,T); V^*)$, $g\in \L^p_\cF(\Omega \times (0,T); \HS(U,H))$ and $u_0 \in \L^p_{\cF_0}(\Omega; (H,V)_{1-\nicefrac{2}{p},p})$ satisfying $\| f \| \leq K$, $\| g \| \leq K$ and $\| u_0 \| \leq K$ in the respective norms.
\end{theorem}
As before, the condition $A = A^*$ can be weakened by a perturbation argument.

To illustrate the power of Theorem \ref{thm:varcomp} we will show that it can be effectively used in the stochastic compactness method to prove global existence of a system of quasi-linear SPDEs (see Theorem \ref{thm:quasi2} below). A technical issue in applying the stochastic compactness method is that the a.s.\ limit, obtained by tightness and the Skorohod embedding theorem, needs to be identified as a solution. The following advantages turn up as a consequence of the uniform estimate we obtain in the proof of Theorem \ref{thm:introExtra}:
\begin{enumerate}[(1)]
\item Tightness in $\Cont([0,T];H)$ (if $V\hookrightarrow H$ is compact);
\item uniform estimates in $\L^p(\Omega; \Cont([0,T];H))$;
\item uniform estimates in $\L^p(\Omega\times (0,T);V)$.
\end{enumerate}
Here we let $u$ run over all solutions with data $(A,B,f,g,u_0)$ as in Theorem \ref{thm:varcomp}. As a consequence of (1) we obtain subsequences which converge in $\Cont([0,T];H)$. As a consequence of (2) and (3) we obtain weak compactness in $\L^p(\Omega\times (0,T);V)$ and uniform integrability of $\{\|u\|_{\Cont([0,T];H)}^2: u\}$ and $\{\|u\|_{\L^2(0,T;V)}^2: u\}$, where $u$ runs over all strong solutions for the given data $(A,B,f,g,u_0)$
The uniform integrability can be effectively combined with Vitali's convergence theorem.

\subsection*{Application to a quasi-linear system of SPDEs}
On an open and bounded set $\Dom\subseteq \R^d$ consider the quasilinear system
	\begin{align}
		\label{eq:qlintro}
		\tag{QLP}
		\left\{\quad
		\begin{aligned}
			\d u^\alpha  &= \big[ \partial_i (a^{\alpha\beta}_{ij}(u) \partial_j u^\beta) + \partial_i \Phi^\alpha_i(u) +\phi^\alpha(u) \big] \d t
			\\  & \qquad \qquad
			+ \sum_{n\geq 1}  \big[ b^{\alpha\beta}_{n,j}(u^\beta) \partial_j u^\beta + g^\alpha_{n}(u) \big] \d w_n,\ \ &\text{ on }\Dom,\\
			u^\alpha &= 0, \ \ &\text{ on }\partial\Dom,
			\\
			u^\alpha(0)&=u^\alpha_{0},\ \ &\text{ on }\Dom.
		\end{aligned}\right.
	\end{align}

We state a special case of Theorem \ref{thm:quasi2}.
 \begin{theorem}[Weak existence for quasi-linear systems]\label{thm:quasi2intro}
		Let $h > 1$. Suppose Assumption~\ref{ass:ql}. Then there is $p_0 > 2$, depending on $T\vee1$, $\lambda$, $\Lambda$, $q$ and $C$ from the assumption such that, if $p\in (2, p_0]$ and $q\in (ph, \infty)$, then, given $u_0\in \L^{p}_{\cF_0}(\Omega;\B^{1-\nicefrac2p}_{2,p,0}(\Dom))\cap \L^q_{\cF_0}(\Omega \times \Dom)$, there exists a weak solution $(\wt{u}, \wt{W}, \wt{\Omega}, \wt{\cF}, \wt{\P}, (\wt{\cF}_t)_{t\geq 0})$ of \eqref{eq:qlintro}.
	\end{theorem}

Besides that we are able to treat highly coupled systems, even in the case of scalar quasilinear equations the result of Theorem \ref{thm:quasi2intro} contains new features:
\begin{enumerate}[(a)]
\item less regularity on the initial data is required;
\item equations on domains with boundary condition can be considered;
\item gradient/transport noise terms are included.
\end{enumerate}

Often when applying the stochastic compactness method only the torus is considered for simplicity. Actually, equations on domains with for instance Dirichlet boundary condition can lead to complications:
\begin{itemize}
\item Bootstrapping regularity can be problematic due the presence of boundary values;
\item It is often unclear in what extrapolation space to formulate compactness/tightness.
\end{itemize}

\subsection*{Open problems}
Given the well-posedness theory and the deterministic extrapolation results it is natural to state the following open problem.
\begin{problem}
Do Theorems \ref{thm:introExtra} and \ref{thm:varcomp} hold without the symmetry condition $A^* = A$ ?
\end{problem}

The well-posedness theory of Lions is usually stated in the setting of nonlinear operators $(A,B)$ satisfying a continuity, monotonicity and coercivity condition. We covered only
the linear setting in Theorems \ref{thm:introExtra} and \ref{thm:varcomp}. This leads to the following natural question.
\begin{problem}
Do Theorems \ref{thm:introExtra} and \ref{thm:varcomp} hold in the setting of nonlinear monotone operators $(A,B)$ ?
\end{problem}
This problem seems to be open even in the deterministic setting.

\subsection*{Organization of the paper}

In Section \ref{sec:main_result1} we present the precise assumptions and state a more general version of Theorem \ref{thm:introExtra}, and we derive Theorem \ref{thm:varcomp} from it. In Section~\ref{sec:L2} we present a result on analytic dependency
of solutions on the equation.
This will be needed later on in order to apply Stein interpolation. In Section \ref{sec:Lp} we discuss a case in which we have $\L^p(\Omega\times (0,T);V)$-estimates for all $p\in [2, \infty)$, which is used as one of the endpoints in the Stein interpolation; the other endpoint is $\L^2(\Omega\times (0,T);V)$.  Finally, in Section \ref{sec:ext} we prove the main result Theorem~\ref{thm:main1}, a more general version of Theorem~\ref{thm:introExtra}.

The main weak existence result on quasi-linear systems is stated in Section \ref{sec:main_result2}. First, a proof in the case of Lipschitz nonlinearities will be given in Section \ref{sec:ql_lip} via Theorem \ref{thm:varcomp} and the stochastic compactness method. This intermediate result has its own value as it imposes even less structural assumptions. The general case is covered in Section \ref{sec:ql_growing}, again via an a priori estimate in $\L^q$ and the stochastic compactness method.

\subsection*{Notation}

\begin{itemize}
\item The notation $a \lesssim_P b$ means that there is a constant $C$ only depending on the parameter $P$ such that $a\leq C b$.

\item $(V, H, V^*)$ is the notation for the Gelfand triple of Hilbert spaces.
\item $(\cdot,\cdot)_{\theta,p}$ stands for real interpolation.
\item $[\cdot,\cdot]_\theta$ stands for complex interpolation.
\item $W$ denotes a cylindrical Brownian motion on the real separable Hilbert space $U$.
\item $\calL_2$ denotes the Hilbert--Schmidt operators.
\end{itemize}
For further unexplained notation the reader is referred to \cite{AVvar,HNVW1,LR}.

\subsubsection*{Acknowledgements}
The authors thank Max Sauerbrey for useful discussions on the stochastic compactness method and the referee for helpful comments and careful reading.

	\section{Main extrapolation result for linear equations}

	\label{sec:main_result1}

	We consider a linear stochastic partial differential equation of the form
	\begin{align}
		\label{eq:sto_heat}
		\tag{P}
		\d u(t) + A(t)u(t) \d t = f(t) \d t + (B(t) u(t) + g(t)) \d W(t), \quad u(0) = u_0.
	\end{align}

	Let us make precise the setting and assumptions for this part.
	\begin{ass}
		\label{ass:SP}
		Fix $0<T<\infty$. Let $U$ be a real separable Hilbert space, and $H$ and $V$ be complex separable Hilbert spaces, where $V \subseteq H$ continuously and densely. Fix a probability space $\Omega$ with filtration $\cF = (\cF_t)_{0 \leq t\leq T}$ and let $W$ be a real-valued $U$-cylindrical Brownian motion adapted to $\cF$. Consider $A\colon \Omega \times (0,T) \to \cL(V, V^*)$ and $B\colon \Omega \times (0,T) \to \cL(V, \HS(U,H))$, where $\HS(U,H)$ denotes the space of Hilbert Schmidt operators from $U$ to $H$, both progressively measurable. Suppose that there exist $\Lambda(A), \Lambda(B), \lambda > 0$ and $M \geq 0$ such that pointwise on $[0,T]\times \Omega$ and $v\in V$ one has
		\begin{align}
			\label{eq:coerc}
			\Re\, \lb Av, v \rb - \frac{1}{2} \| Bv \|_{\HS(U,H)}^2 \geq \lambda \| v \|_V^2 - M \| v \|_H^2,
		\end{align}
		and
		\begin{align}
			\| Av \|_{V^*} \leq \Lambda(A) \| v \|_V \quad \& \quad \| Bv \|_{\HS(U,H)} \leq \Lambda(B) \| v \|_V.
		\end{align}
		Put $\Lambda = \Lambda(A) + \Lambda(B)$.
	\end{ass}

	The main result of this part is the following.

	\begin{theorem}
		\label{thm:main1}
		Suppose that Assumption~\ref{ass:SP} holds and $\lb Au, v \rb = \lb Av, u\rb $ pointwise in $[0,T]\times \Omega$. Then there exists a $p_0>2$ such that, for all $p\in [2, p_0]$, $f\in \L^p_\cF(\Omega \times (0,T); V^*)$, $g\in \L^p_\cF(\Omega \times (0,T); \HS(U,H))$ and $u_0 \in \L^p_{\cF_0}(\Omega; (H,V)_{1-\nicefrac{2}{p},p})$, problem~\eqref{eq:sto_heat} admits a unique solution $u \in \L^p(\Omega \times (0,T); V)$, and for any $\theta \in [0,\nicefrac{1}{2})$ there is a constant $K_\theta$ depending only on $\Lambda(A)$, $\Lambda(B)$, $\lambda$, $M$, $T\vee 1$, $\theta$ such that
		\begin{align}
			\| u \|_{\L^p(\Omega; \H^{\theta,p}(0,T; [H, V]_{1-2\theta}))} \leq K_\theta C_{u_0,f,g},
		\end{align}
		where
		\begin{align}
			C_{u_0,f,g} = \| u_0 \|_{\L^p(\Omega; (H,V)_{1-\nicefrac{2}{p},p})} + \| f \|_{\L^p(\Omega \times (0,T); V^*)} + \| g \|_{\L^p(\Omega \times (0,T); \HS(U,H))}.
		\end{align}
		Also, there is a constant $K_0$ depending only on $\Lambda(A)$, $\Lambda(B)$, $\lambda$, $M$, $T\vee 1$, and $p$ such that
		\begin{align}
			\| u \|_{\L^p(\Omega; \Cont([0,T]; (H, V)_{1-\nicefrac{2}{p},p}))} \leq K_0 C_{u_0,f,g}.
		\end{align}
		Moreover, for $\theta \in (\nicefrac{1}{p},\nicefrac{1}{2})$, there is a constant $K_\theta$ depending only on $\Lambda(A)$, $\Lambda(B)$, $\lambda$, $T\vee 1$, and $\theta$ such that
		\begin{align}
			\| u \|_{\L^p(\Omega; \Cont^{\theta-\nicefrac{1}{p}}([0,T]; [H, V]_{1-2\theta}))} &\leq K_\theta C_{u_0,f,g}.
		\end{align}
	\end{theorem}

The symbol $\H^{\theta,p}$ denotes the (vector-valued) Bessel potential space (see \cite{HNVW1,HNVW3}), and the subscripts $\cF$ and $\cF_0$ refer to the subspaces of $\L^p$ of $\cF$-progressively measurable and $\cF_0$-adapted functions.

	\begin{remark}[Real problems]
		In virtue of complexification, the result does also apply to problems over real Hilbert spaces. We are going to use this fact later on.
	\end{remark}

As a consequence of Theorem \ref{thm:main1} we can already prove the compactness/tightness result of Theorem \ref{thm:varcomp}.
\begin{proof}[Proof of Theorem \ref{thm:varcomp}]
Let $S$ denote the set of solutions described in Theorem \ref{thm:varcomp}. Fix $\theta\in (1/p,1/2)$. By Theorem \ref{thm:main1}, $S$ is bounded in $\L^p(\Omega; \Cont^{\theta-\nicefrac{1}{p}}([0,T]; [H, V]_{1-2\theta}))$ by $K_{\theta} C_{u_0, f, g}$. Let $\varepsilon>0$.
Choose $R>0$
such that $K_{\theta}^p C_{u_0, f, g}^p R^{-p} \leq \varepsilon$. Then for all $u\in S$,
\[\P(\|u\|_{\Cont^{\theta-\nicefrac{1}{p}}([0,T]; [H, V]_{1-2\theta})}\geq R)\leq R^{-p}\mathbb{E}\|u\|_{\Cont^{\theta-\nicefrac{1}{p}}([0,T]; [H, V]_{1-2\theta})}^p\leq \varepsilon.\]
By the vector-valued Arz\'ela--Ascoli theorem \cite[Theorem III.3.1]{Lang}, the embedding
\[\Cont^{\theta-\nicefrac{1}{p}}([0,T]; [H, V]_{1-2\theta})\hookrightarrow \Cont([0,T];H)\]
is compact, and therefore $\{\calL(u):u\in S\}$ is tight.
\end{proof}

	\section{$\L^2$-estimates and analytic dependency}
	\label{sec:L2}

	The central finding of this section is that solutions to a complex family of variational problems have analytic dependence provided this was the case for the complex family, see Proposition~\ref{prop:analytic} and Corollary \ref{cor:analytic}. Before, we will briefly discuss well-posedness of~\eqref{eq:sto_heat} in a complex setting.

	To ease notation, we introduce data and solution spaces for the variational setting.

	\begin{definition}
		\label{def:spaces1}
		Define the data spaces $\E \coloneqq \L^2_\cF(\Omega \times (0,T); V^*)$ and $\Egam \coloneqq \L^2_\cF(\Omega \times (0,T); \HS(U,H))$, and the solution space $\V \coloneqq \L^2_\cF(\Omega \times (0,T); V) \cap \L^2_\cF(\Omega; \Cont([0,T];H))$, where the subscript $\cF$ indicates the subspace of progressively measurable functions.
	\end{definition}

	\subsection{A (complex) variational setting}
	\label{subsec:variation}

	The following variational well-posedness result is well-known in its real formulation~\cite{LR}. The complex version follows by \enquote{forgetting} the complex structure. For instance, we can interpret a complex vector space $H$ as a real vector space if we equip it with the inner product $(\cdot | \cdot)_{H_\R} = \Re (\cdot | \cdot)_H$ and so on.

	\begin{proposition}
		\label{prop:lions}
		Let $f\in \E$, $g\in \Egam$. Suppose Assumption~\ref{ass:SP} holds. Then~\eqref{eq:sto_heat} with $u_0 \in \L^2_{\cF_0}(\Omega;H)$ has a unique solution $u \in \V$, and there exists a constant $C_L$ depending on $\Lambda$, $\lambda$, $M$, and $T\vee 1$, such that
		\begin{align}
			\label{eq:lions}
			\| u \|_{\V} \leq C_L \bigl( \| u_0 \|_{\L^2(\Omega;H)} + \| f \|_{\E} + \| g \|_{\Egam} \bigr).
		\end{align}
	\end{proposition}

	\subsection{Analytic dependence of solutions}
	\label{subsec:analytic}

	In the following result we show analytic dependence of solutions.

	\begin{proposition}
		\label{prop:analytic}
		Let $O \subseteq \C$ be open and for $z\in O$ let $A_z$ and $B_z$ be operator functions that satisfy Assumption~\ref{ass:SP} uniformly in $z$ and that depend analytically on $z$ in the uniform operator topology. Furthermore, let $f \in \E$, $g \in \Egam$ and $u_0\in H$. Then, for fixed $z$, the unique solution $u_z$ of the problem
		\begin{align}
			\label{eq:analytic}
			\d u(t) + A_z(t)u(t) \d t = f(t) \d t + (B_z(t) u(t) + g(t)) \d W(t), \quad u(0) = u_0,
		\end{align}
		gives rise to an analytic function $O \ni z \mapsto u_z \in \V$.
	\end{proposition}

	\begin{proof}
		Since we consider differences below, we have zero initial data in the equations in this proof. For fixed $z$, Proposition~\ref{prop:lions} yields a unique solution $u_z$ of~\eqref{eq:analytic}.

		\textbf{Step 1}: $O \ni z\mapsto u_z \in \V$ is continuous.

		Let $z \in O$ and $h$ small enough so that $z+h \in O$. The function $v \coloneqq u_{z+h} - u_z$ is the unique solution of
		\begin{align}
				\d v(t) + A_{z+h} v(t) \d t = -(A_{z+h}-A_z) u_z \d t + (B_{z+h} v(t) + (B_{z+h}-B_z) u_z)\d W(t).
		\end{align}
		Therefore, by Proposition~\ref{prop:lions},
		\begin{align}
			\| v \|_{\V} &\leq C_L \bigl( \| (A_{z+h}-A_z) u_z \|_{\E} + \| (B_{z+h} - B_z) u_z \|_{\Egam} \bigr).
		\end{align}
		Hence, continuity of $A_z$ and $B_z$ yield the claim.

		\textbf{Step 2}: $O \ni z\mapsto u_z \in \V$ is analytic.

		Write $D^h$ for difference quotients of $u$, $A$, and $B$ with respect to $z$, for example $D^h u = \frac{u(z+h)-u(z)}{h}$.
		One has that $D^h u$ is the unique solution of
		\begin{align}
			\d v(t) + A_z v \d t = -(D^h A_z) u_{z+h} \d t + ( B_z v + (D^h B_z) u_{z+h} ) \d W(t).
		\end{align}
		We show that $D^h u$ is a Cauchy sequence in $\V$. Then, by the very definition of complex differentiability, $z \mapsto u_z$ is analytic. To this end, consider $h_1 \neq h_2$. The difference $w \coloneqq D^{h_1} u - D^{h_2} u$ solves
		\begin{align}
			\d w(t) + A_z w \d t = -\Bigl( \bigl(D^{h_1} A - D^{h_2} A \bigr) u_{z+h_1} + (D^{h_2} A) (u_{z+h_1} - u_{z+h_2}) \Bigr) \d t \\
			+ \Bigl( B_z w + \bigl( D^{h_1} B - D^{h_2} B \bigr) u_{z+h_1} + (D^{h_2} B) (u_{z+h_1} - u_{z+h_2}) \Bigr) \d W(t).
		\end{align}
		A calculation using Proposition~\ref{prop:lions} shows
		\begin{align}
			&\| D^{h_1}u - D^{h_2} u \|_{\V} \\
			\leq{} &C_L^2 \Bigl( \| D^{h_1}A - D^{h_2}A \|_{\cL(V, V^*)} + \| D^{h_1} B - D^{h_2} B \|_{\cL(V, \HS(U,H))} \Bigr) (\| f \|_{\E} + \| g \|_{\Egam}) \\
			&\quad+ C_L \Bigl(\| D^{h_2} A \|_{\cL(V, V^*)} + \| D^{h_2} B \|_{\cL(V, \HS(U,H))} \Bigr) \| u_{z+h_1} - u_{z+h_2} \|_{\V}.
		\end{align}
		First, the difference quotients of $A$ and $B$ are Cauchy sequences for $A$ and $B$ are holomorphic. Second, the difference quotients of $A$ and $B$ are, again by analyticity, uniformly bounded. Third, $\| u_{z+h_1} - u_{z+h_2} \|_{\V}$ tends to zero by Step~1. We conclude that $D^h u$ is a Cauchy sequence as claimed.
	\end{proof}

\begin{corollary}\label{cor:analytic}
		Let $O \subseteq \C$ be open and for $z\in O$ let $A_z$ and $B_z$ be operator functions that satisfy Assumption~\ref{ass:SP} uniformly in $z$ and that depend analytically on $z$ in the uniform operator topology. Furthermore, let $f:O\to \E$, $g: O\to \Egam$ and $u_0:O\to H$ be analytic. Then, for fixed $z$, the unique solution $u_z$ of the problem
		\begin{align}\label{eq:uzcorana}
			\d u(t) + A_z(t)u(t) \d t = f_z(t) \d t + (B_z(t) u(t) + g_z(t)) \d W(t), \quad u(0) = u_0,
		\end{align}
		gives rise to an analytic function $O \ni z \mapsto u_z \in \V$.\end{corollary}
\begin{proof}
From Proposition \ref{prop:analytic} it follows that $S_z:O\to \calL(\E\times \Egam\times H,\V)$ given by $S_z = u_z$, where $u_z$ is the solution to \eqref{eq:analytic}, is analytic in the strong operator topology. Therefore, $S$ is analytic in the uniform operator topology (see \cite[Proposition A.3]{ABHN}). Now the analyticity follows as the solution to \eqref{eq:uzcorana} is given by $S_z (f_z, g_z, u_{0,z})$ and is the composition of analytic functions.
\end{proof}

	\section{$\L^p$-estimates for small perturbations of the autonomous case}
	\label{sec:Lp}

	In this section, we establish $\L^p$-theory by means of a perturbation result with the autonomous case. The quantities in the assumption of Theorem~\ref{thm:perturb} look cumbersome at first glance, but we will see in Section~\ref{sec:ext} that they are, in fact, very closely tied to the concept of coercivity.

	Throughout this section, fix some $p>2$. We extend Definition~\ref{def:spaces1}.
	\begin{definition}
		\label{def:spaces2}
		Put $\Ep \coloneqq \L^p_\cF(\Omega \times (0,T); V^*)$ and $\Epgam \coloneqq \L^p_\cF(\Omega \times (0,T); \HS(U,H))$ for the data spaces and $\Vp \coloneqq \L^p_\cF(\Omega \times (0,T); V) \cap \L^p_\cF(\Omega; \Cont([0,T]; (H,V)_{1-\nicefrac{2}{p},p}))$ for the solution space.
	\end{definition}

	Let us introduce a reference operator for the perturbation argument.

	\begin{definition}
		Consider the operator $A_0 \colon V \to V^*$ given by $\lb A_0(u), v \rb \coloneqq (u | v)_V$.
	\end{definition}

	\begin{remark}
		The operator $A_0$ is invertible, positive and self-adjoint.
	\end{remark}

The following deterministic maximal $L^p$-regularity result holds.
	\begin{lemma}\label{lem:A0apriori}
		There is a constant $C_p$ such that, for all $\mu > 0$, $f \in \L^p(0,T; V^*)$, and $u$ a strong solution to $u' + (\mu A_0) u = f$ with $u(0) = 0$, one has the a-priori estimate
		\begin{align}
			\mu \| u \|_{\L^p(0,T; V)} \leq C_p \| f \|_{\L^p(0,T; V^*)}.
		\end{align}
	\end{lemma}
\begin{proof}
The fact that $A_0$ has maximal $L^p$-regularity on $\R_+$ follows from de Simon's result (see \cite[Corollary 17.3.8]{HNVW3}). From the proof of \cite[Theorem 17.2.26]{HNVW3} it follows that $\mu A_0$ has maximal $L^p$-regularity on $\R_+$ with the same constant $C_p$. Therefore, by \cite[Lemma 17.2.16]{HNVW3}), $\mu A_0$ has maximal $L^p$-regularity on $(0,T)$ with the same constant $C_p$. This gives the desired bound with constant independent of $\mu$
\end{proof}

Next we will present a similar result in the stochastic setting, but without the optimal scaling in $\mu$ as this will not be needed later.
	\begin{lemma}
		\label{lem:A0SMR}
		For any $\mu > 0$, the operator $\mu A_0$ admits stochastic maximal $\L^p$-regularity, that is to say, there is $c(p, \mu) > 0$ such that for all $f\in \Ep$ and $g\in \Epgam$ there is a unique strong solution $u \in \V$ to
		\begin{align}
			\d u(t) + \mu A_0 u \d t = f(t) \d t + g(t) \d W(t), \qquad u(0) = 0,
		\end{align}
		satisfying the estimate
		\begin{align}
			\| u \|_{\Vp} \leq c(p, \mu) \bigl(\| f\|_{\Ep} + \| g \|_{\Epgam} \bigr).
		\end{align}
	\end{lemma}

	\begin{proof}
	Let $\mu > 0$. By linearity and Lemma \ref{lem:A0apriori}  it suffices to consider $f=0$. As $\mu A_0$ is positive and self-adjoint on $V^*$, it is also positive and self-adjoint on $H$ (see \cite[Proposition 1.24]{Ouh}). Therefore, it has a bounded $\H^\infty$-calculus with constant $1$
(see \cite[Proposition 10.2.23]{HNVW2}). Moreover, $H$ is, as a separable Hilbert space, isomorphic to $\L^2(\R)$. Therefore, from the definition of $A_0$ and from~\cite{vNVW} we obtain
\[\mu^{1/2} \|u\|_{\Vp} = \|(\mu A_0)^{1/2} u\|_{\L^p(\Omega\times\R_+;H)} \leq c(p) \|g\|_{\L^p(\Omega\times\R_+;\calL_2(U,H))} = c(p) \|g\|_{\Epgam}.\]
	\end{proof}

	\begin{theorem}
		\label{thm:perturb}
		Let $f\in \Ep$ and $g\in \Epgam$. Employ Assumptions~\ref{ass:SP} with $M=0$ and recall the constant $\Lambda(B)$. Suppose that there are $\eps, \delta > 0$ such that
		\begin{enumerate}
			\item for some $\mu > 0$ one has $C_p \mu^{-1} \| A - \mu A_0 \|_{\cL(V,V^*)} \leq 1-\eps$, where $C_p$ is the constant from Lemma~\ref{lem:A0apriori}, and
			\item $c \Lambda(B) \leq 1-\delta$, where $c=c(\Lambda,\mu,T,\eps,p)$. \label{it:perturb_B}
		\end{enumerate}
		Then there exists a unique strong solution $u\in \Vp$ to~\eqref{eq:sto_heat} with $u_0 = 0$ satisfying the estimate
		\begin{align}
			\| u \|_{\Vp} \leq \frac{c}{\delta} \bigl( \| f \|_{\Ep} + \| g \|_{\Epgam} \bigr).
		\end{align}
	\end{theorem}

	\begin{proof}
		We tacitly impose the initial condition $u(0)=0$ whenever we speak about equations in this proof. First, we investigate the problem with $B=0$, that is to say, we want to show that there exists a unique strong solution $u\in \Vp$ to the stochastic problem
		\begin{align}
			\label{eq:Llin}
			\d u(t) + A(t)u(t) \d t = f(t) \d t + g(t) \d W(t)
		\end{align}
		satisfying the estimate
		\begin{align}
			\label{eq:apriori}
			\| u \|_{\Vp} \leq c (\| f \|_{\Ep} + \| g \|_{\Epgam})
		\end{align}
		with a constant $c$ depending only on $\eps$, $\Lambda$, $\mu$, $T\vee 1$, and $p$. This constant is the constant postulated in hypothesis~\ref{it:perturb_B} of the theorem.

		\textbf{Step 1}: Reduction to an a-priori estimate.

		For $\theta \in [0,1]$ consider $A_\theta \coloneqq (1-\theta) \mu A_0 + \theta A$. Since $A_\theta - \mu A_0 = \theta (A - \mu A_0)$, one has $$C_p \mu^{-1} \| A_\theta - \mu A_0 \|_{\cL(V, V^*)} = C_p \mu^{-1} \theta \| A - \mu A_0 \|_{\cL(V, V^*)} \leq 1 - \eps.$$

		Therefore, in the light of a stochastic version of the method of continuity (see \cite[Prop. 3.10]{PV19}), it suffices to show the a priori estimate~\eqref{eq:apriori} for any strong solution $u$ of~\eqref{eq:Llin}.

		\textbf{Step 2}: Show the a-priori estimate.

		Let $u$ a strong solution of~\eqref{eq:Llin}. By Lemma~\ref{lem:A0SMR} there is a strong solution $v \in \Vp$ of $$\d v(t) + (\mu A_0) v(t) \d t = f(t) \d t + g(t) \d W(t)$$ satisfying the estimate
		\begin{align}
			\label{eq:est_v}
			\| v \|_{\Vp} \leq c(p, \mu) (\| f\|_{\Ep} + \| g \|_{\Epgam}).
		\end{align}
		Define $w \coloneqq u - v \in \Vp$. Fix $\omega \in \Omega$. Use the shorthand notation $w_\omega(t) = w(\omega,t)$, and so on. By construction, $w_\omega$ is almost surely a strong solution to the deterministic problem $$\partial_t w_\omega + (\mu A_0) w_\omega = -(A_\omega - \mu A_0) u_\omega.$$ Lemma~\ref{lem:A0apriori} lets us compute
		\begin{align}
			\mu \| w_\omega \|_{\L^p(0,T; V)} &\leq C_p \| (A_\omega - \mu A_0) u_\omega \|_{\L^p(0,T; V^*)} \\
			&\leq C_p \| A_\omega - \mu A_0 \|_{\cL(V,V^*)} \Bigl( \| v_\omega \|_{\L^p(0,T; V)} + \| w_\omega \|_{\L^p(0,T; V)} \Bigr).
		\end{align}
		Divide this by $\mu$ and use the assumption to give
		\begin{align}
			\| w_\omega \|_{\L^p(0,T; V)} &\leq C_p \mu^{-1} \| (A_\omega - \mu A_0) \|_{\cL(V,V^*)} \Bigl( \| v_\omega \|_{\L^p(0,T; V)} + \| w_\omega \|_{\L^p(0,T; V)} \Bigr) \\
			&\leq  (1-\eps) \| v_\omega \|_{\L^p(0,T; V)} + (1-\eps) \| w_\omega \|_{\L^p(0,T; V)}.
		\end{align}
		Absorb the second term on the right-hand side into the left-hand side to deduce
		\begin{align}
			\| w_\omega \|_{\L^p(0,T; V)} \leq \frac{1-\eps}{\eps}\| v_\omega \|_{\L^p(0,T; V)}.
		\end{align}
		According to~\cite[Step~1 in Thm.~3.9]{PV19}, $u$ is progressively measurable. Hence, averaging the last bound over $\Omega$ yields
		\begin{align}
			\| w \|_{\L^p(\Omega \times (0,T); V)} \leq \frac{1-\eps}{\eps} \| v \|_{\L^p(\Omega \times (0,T); V)}.
		\end{align}
		In conjunction with $u=v+w$ and~\eqref{eq:est_v} we obtain in summary
		\begin{align}
			\label{eq:grad_est}
			\| u \|_{\L^p(\Omega \times (0,T); V)} &\leq \| v \|_{\L^p(\Omega \times (0,T); V)} + \| w \|_{\L^p(\Omega \times (0,T); V)} \leq \frac{c(p,\mu)}{\eps} \bigl( \| f\|_{\Ep} + \| g \|_{\Epgam} \bigr).
		\end{align}

		\textbf{Step 3}: Including $B$.

		As is standard (see~\cite[Prop.~3.10]{PV19}, for instance), we solve the problem~\eqref{eq:sto_heat} with $u_0 = 0$,
		\begin{align}
			\label{eq:Lnonlin}
			\d u(t) + A(t)u(t) \d t = f(t) \d t + (B(t)u(t) + g(t)) \d W(t),
		\end{align}
		using a fixed-point argument. For $\Phi \in \Vp$, replace $B u$ by $B \Phi$ in~\eqref{eq:Lnonlin}. This problem possesses a unique solution in $\Vp$ in virtue of the first part of the proof. Write $R(\Phi)$ for it. To show unique solvability of~\eqref{eq:Lnonlin}, it suffices to show that $R$ is a strict contraction.

		Let $\Phi_1, \Phi_2 \in \Vp$. By linearity, $R(\Phi_1) - R(\Phi_2)$ is the unique strong solution to the problem $$\d u(t) + A(t)u(t) \d t = B(t)(\Phi_1(t) - \Phi_2(t)) \d W(t).$$ Hence,~\eqref{eq:apriori} with $f=0$ and $g = B (\Phi_1 - \Phi_2)$ gives
		\begin{align}
			\| R(\Phi_1) - R(\Phi_2) \|_{\Vp} &\leq C \| B (\Phi_1 - \Phi_2) \|_{\Epgam} \leq C \Lambda(B) \| \Phi_1 - \Phi_2 \|_{\Vp}.
		\end{align}
		Thus, $R$ is a strict contraction by the assumption $C \Lambda(B) \leq 1-\delta$.

		For the estimate, apply the a-priori estimate~\eqref{eq:apriori} once more to find
		\begin{align}
			\| u \|_{\Vp} &\leq C( \| f \|_{\Ep} + \| g \|_{\Epgam} + \| Bu \|_{\Epgam}) \leq C( \| f \|_{\Ep} + \| g \|_{\Epgam}) + (1-\delta) \| u \|_{\Vp}.
		\end{align}
		We can absorb the second term of the right-hand side into the left-hand side to conclude.
	\end{proof}

	\section{Extrapolation of integrability in time}
	\label{sec:ext}

	Our main result follows from an interpolation argument using abstract Stein interpolation~\cite{AbstractStein}. For details on interpolation the reader is  referred to \cite{BeLo,HNVW1,Tr1}. The last two sections constitute the endpoint cases of this interpolation. Before we can conclude in Section~\ref{subsec:stein}, we have to construct an analytic family of auxiliary problems first.

	\begin{ass}
		\label{ass:symmetric}
		For fixed $t$, the form $V\times V\ni (u,v) \mapsto \lb A(\omega,t)u, v \rb$ is hermitian almost surely. To ease notation, write $a(u,v) \coloneqq a_{\omega,t}(u,v) \coloneqq \lb A(\omega,t)u, v \rb$.
	\end{ass}
This symmetry assumption is used in a crucial way in Lemmas \ref{lem:dist_A-B} and \ref{lem:elliptic}.

	\begin{conv}
		We ignore the dependence on $(\omega,t)$ in our notation. For instance, we let $A$ denote the operator $V\to V^*$ given by $\lb Au, v \rb \coloneqq \lb A(\omega, t)u, v \rb$. Also, we put $a(u,v) \coloneqq \lb Au, v \rb$ and $a[u] \coloneqq a(u,u)$.
	\end{conv}

	The following lemma translates upper bounds on the diagonal of $V\times V$ to the whole form. That the optimal constant is $1$ is a consequence of Assumption~\ref{ass:symmetric}, but is actually not needed later. A simpler polarization argument without Assumption \ref{ass:symmetric} would lead to an extra factor $2$.

	\begin{lemma}
		\label{lem:diag}
Suppose that Assumption \ref{ass:symmetric} holds. If for some $c>0$ one has $|a[u]| \leq c \| u \|_V^2$, then also $|a(u,v)| \leq c \| u \|_V \| v \|_V$ for all $u,v \in V$.
	\end{lemma}

	\begin{proof}
		Let $u,v \in V$. We can of course assume $u \neq 0 \neq v$. There is a complex number $z$ with $|z|=1$ such that $za(u,v)$ is real and positive. By the polarization identity one has
		\begin{align}
			\label{eq:pola}
			z a(u,v) = a(zu, v) = \frac{1}{4} \bigl(a[zu+v] - a[zu-v] + \i a[zu+\i v] - \i a[zu-\i v] \bigr).
		\end{align}
		Since $a$ is hermitian, $a[u]$ is real for any $u \in V$. Hence, taking the real part of~\eqref{eq:pola} gives $$|a(u,v)| = z a(u,v) = \frac{1}{4} \bigl(a[zu+v] - a[zu-v] \bigr).$$ Expanding the quadratic forms in conjunction with the bound on the diagonal leads to
		\begin{align}
			\label{eq:pola2}
			|a(u,v)| = \frac{1}{4} \bigl( 2 a[zu] + 2 a[v] \bigr) \leq \frac{c}{2} \bigl( \| u \|_V^2 + \| v \|_V^2 \bigr).
		\end{align}
It remains to use $|a(u,v)| = |a(\lambda u,\lambda^{-1}v)|$, apply \eqref{eq:pola2}, and minimize over $\lambda>0$.
	\end{proof}

	\subsection{Complex family of auxiliary problems}
	\label{subsec:aux}

	\begin{lemma}
		\label{lem:dist_A-B}
		Suppose that Assumption~\ref{ass:SP} holds with $M=0$ and that Assumption \ref{ass:symmetric} holds.
		Put $\mu = \Lambda(A)$ and $\rho = \frac{\lambda}{\Lambda(A)}$. Then one has $$\Bigl| \mu^{-1} \Bigl( a[v] - \frac{1}{2} \| B v \|_{\HS(U,H)}^2 \Bigr) - \| v \|_V^2 \Bigr| \leq (1-\rho) \| v\|_V^2.$$
	\end{lemma}
	\begin{proof}
		To ease notation, put $c[v] \coloneqq a[v] - \frac{1}{2} \| B v \|_{\HS(U,H)}^2$. By Assumptions~\ref{ass:SP} (with $M=0$) and~\ref{ass:symmetric} one has $\lambda \| v \|_V^2 \leq c[v] \leq \Lambda(A) \| v \|_V^2$. Let $s \geq 0$. Assumption~\ref{ass:symmetric} implies that $s c[v] - \| v \|_V^2$ is real, hence
		\begin{align}
		\label{eq:max}
			| s c[v] - \| v \|_V^2| = \max( s c[v] - \| v \|_V^2, \| v \|_V^2 - s c[v]).
		\end{align}
		Consider the first term. Since $$s c[v] - \| v \|_V^2 \leq (s \Lambda(A) - 1) \| v \|_V^2,$$ the maximum in~\eqref{eq:max} coincides with the second term for any choice $s \leq \frac{1}{\Lambda(A)}$. On the other hand, $$\| v \|_V^2 - s c[v] \leq (1-s \lambda)\| v\|_V^2,$$ so the right-hand side of the last estimate is minimal for the maximal admissible choice $s = (\Lambda(A))^{-1}$. Put $\mu = s^{-1} = \Lambda(A)$ and $\rho = s\lambda =\frac{\lambda}{\Lambda(A)}$. In summary, we obtain
		\begin{align}
			| \mu^{-1} c[v] - \| v \|_V^2 | \leq (1-\rho) \| v \|_V^2. &\qedhere
		\end{align}
	\end{proof}

	The same calculation but applied to $a[v]$ instead of $a[v] - \frac{1}{2} \| B v \|_{\HS(U,H)}^2$ in conjunction with Lemma~\ref{lem:diag} gives the following.
	\begin{corollary}
		\label{cor:dist_A}
		Suppose that Assumption~\ref{ass:SP} holds with $M=0$ and that Assumption \ref{ass:symmetric} holds.	For the same $\mu$ and $\rho$ as in Lemma~\ref{lem:dist_A-B} one has $\| \mu^{-1} A - A_0 \|_{\cL(V,V^*)} \leq 1-\rho$.
	\end{corollary}
	We use $\mu$ and $\rho$ from the lemma to define complex perturbations of $A$ and $B$.
	\begin{definition}
		\label{def:perturb}
		With $\mu$ and $\rho$ from Lemma~\ref{lem:dist_A-B}, fix numbers $0<r<1<R$ satisfying
		\begin{enumerate}
			\item $R (1-\rho) < 1$, and \label{it:R}
			\item $r \min(C_p (1-\rho), c \Lambda(B)) < 1$, where $C_p$ is the constant from Lemma~\ref{lem:A0apriori} and $c$ is the constant from Theorem~\ref{thm:perturb}. \label{it:r}
		\end{enumerate}
		Moreover, with the holomorphic function $F(z) = r \e^{z \log(R/r)}$ defined on the open unit strip $S = \{z\in \C: 0<\Re(z)<1\}$, define
		\begin{align}
			A_z \coloneqq \mu \bigl(F(z) (\mu^{-1} A - A_0) + A_0\bigr) \quad \& \quad B_z \coloneqq F(z)^\frac{1}{2} B.
		\end{align}
	\end{definition}
	\begin{lemma}
		\label{lem:elliptic}
Suppose that Assumption~\ref{ass:SP} holds with $M=0$ and that Assumption \ref{ass:symmetric} holds.	Then the coercivity condition~\eqref{eq:coerc} in Assumption~\ref{ass:SP} is satisfied for $A_z$ and $B_z$, uniformly in $z\in \overline{S}$. To be more precise, the implied constant depends only on $\lambda$, $\Lambda(A)$, and $R$.
	\end{lemma}
	\begin{proof}
		Write $\| \cdot \|_{\HS} \coloneqq \| \cdot \|_{\HS(U,H)}$. Recall $\mu = \Lambda(A)$ from Lemma~\ref{lem:dist_A-B}. Consequently, $\| \mu^{-1} A \|_{\cL(V,V^*)} \leq 1$. Since $a$ is hermitian by Assumption \ref{ass:symmetric}, it follows that $\mu^{-1} a[v] - \| v \|_V^2$ is real and non-positive. Recall $$\mu^{-1} A_z = F(z)(\mu^{-1} A - A_0) + A_0.$$ Calculate for $\| v \|_V = 1$ that
		\begin{align}
			\Re \mu^{-1} \bigl( a_z[v] - \frac{1}{2} \| B_z(v) \|_{\HS}^2 \bigr)
			&={} 1 + \bigl(\Re F(z)\bigr) (\mu^{-1} a[v] - \| v \|_V^2) -  \frac{|F(z)|}{2\mu} \| B(v) \|_{\HS}^2 \\
			&\geq{} 1 + |F(z)| (\mu^{-1} a[v] - \| v \|_V^2) -  \frac{|F(z)|}{2\mu} \| B(v) \|_{\HS}^2 \\
			&={} 1 + |F(z)| \Bigl(\mu^{-1} \bigl( a[v] - \frac{1}{2} \| B(v) \|_{\HS}^2 \bigr) - \| v\|_V^2 \Bigr) \\
			&\geq{} 1 - |F(z)| \Bigl|\mu^{-1} \bigl(a[v] - \frac{1}{2} \| B(v) \|_{\HS}^2\bigr) - \| v\|_V^2 \Bigr|,
		\end{align}
		where we exploited in the first line that $\mu^{-1} a[v] - \| v \|_V^2$ is real and in the second line that it is non-positive.
		Therefore, Lemma~\ref{lem:dist_A-B} reveals
		\begin{align}
			\Re \mu^{-1} \bigl( a_z[v] - \frac{1}{2} \| B_z(v) \|_{\HS}^2 \bigr) \geq 1 - |F(z)|(1-\rho),
		\end{align}
		so that the right-hand side is strictly positive in virtue of Definition~\ref{def:perturb}~\ref{it:R} and since $|F(z)|\leq r\e^{\log(R/r)} = R$. Eventually, multiplying by $\mu$ completes the proof.
	\end{proof}
	Now consider on the unit strip $S$ the complex family of problems
	\begin{align}
		\label{eq:sto_heat_z}
		\tag{$P_z$}
		\d u(t) + A_z(t)u(t) \d t = f(t) \d t + (B_z(t) u(t) + g(t)) \d W(t)
	\end{align}
	with initial condition $u(0)=0$.

	By construction, the mappings $z\mapsto A_z$ and $z\mapsto B_z$ are analytic, and Lemma~\ref{lem:elliptic} assures that Assumption~\ref{ass:SP} is satisfied uniformly in $z$. We conclude with Proposition~\ref{prop:analytic} that $z\mapsto u_z \in \V$, where $u_z$ is the unique solution of~\eqref{eq:sto_heat_z}, is analytic.

	\subsection{Conclusion by Stein interpolation}
	\label{subsec:stein}

	The following proposition is the basis for our main result.

	\begin{proposition}
		\label{prop:main}
		Under Assumptions~\ref{ass:SP} with $M=0$ and~\ref{ass:symmetric} there exists a $p_0>2$ such that, for all $p\in [2, p_0]$, $f\in \E^p$ and $g\in \Epgam$, problem~\eqref{eq:sto_heat} with $u_0 = 0$ admits a unique solution $u \in \V^p$, and there is a constant $C$ depending only on $\Lambda(A)$, $\Lambda(B)$, $\lambda$, $T\vee 1$ such that
		\begin{align}
			\| u \|_{\V^p} \leq C \bigl( \| f \|_{\E^p} + \| g \|_{\Epgam} \bigr).
		\end{align}
	\end{proposition}

	\begin{proof}
		\textbf{Step 1}: Applying abstract Stein interpolation to~\eqref{eq:sto_heat_z}.

		For $z\in \overline{S}$ define an operator $T_z \colon \E \times \Egam \to \V$ which maps a data pair $(f,g)$ to the unique solution $u$ of~\eqref{eq:sto_heat_z}. By linearity of the equation, $T_z$ is linear and the family of operator $(T_z)_{z\in \overline{S}}$ is uniformly bounded according to Proposition~\ref{prop:lions}.

		Fix some $q>2$. We apply the results of the last section with $p$ replaced by $q$. Let $(f,g) \in \Eq \times \Eqgam \subseteq \E \times \Egam$. We claim that $T_{\i t}(f,g) \in \Vq$ uniformly. Indeed, this follows from Theorem~\ref{thm:perturb} if we check its smallness assumptions. Hypothesis~\ref{it:r} follows immediately from Definition~\ref{def:perturb}. Moreover, since $A_0 - \mu^{-1} A_z = F(z)(A_0 - \mu^{-1} A)$, it follows from Corollary~\ref{cor:dist_A} that
		\begin{align}
			C_q \| A_0 - \mu^{-1} A_{\i t} \|_{\cL(V,V^*)} = C_q |F(\i t)| \| A_0 - \mu^{-1} A \|_{\cL(V,V^*)} \leq r C_q (1-\rho).
		\end{align}
		Therefore, by choice of $r$ in Definition~\ref{def:perturb}, $C_q \| A_0 - \mu^{-1} A_{\i t} \|_{\cL(V,V^*)} < 1$ uniformly in $z=\i t$.

		Finally, we have argued in Section~\ref{subsec:analytic} that for fixed $(f,g)$ the map $S \ni z \mapsto T_z(f,g) \in \V$ is holomorphic. In summary, this allows us to invoke abstract Stein interpolation~\cite{AbstractStein} to deduce
		\begin{align}
			\| u_\theta \|_{\V^{q_\theta}} \leq C_\theta \bigl( \| f \|_{\E^{q_\theta}} + \| g \|_{\mathbb{E}^{q_\theta}_{\frac12}} \bigr), \quad \theta \in (0,1), f\in \Eq, g\in \Eqgam,
		\end{align}
		where $\frac{1}{q_\theta} = \frac{1-\theta}{2} + \frac{\theta}{q}$ and $C_\theta = C_L^{1-\theta} (c/\delta)^\theta$, $u_\theta$ is the unique solution to~\eqref{eq:sto_heat_z} with $z=\theta$, and where we use the constants from Proposition~\ref{prop:lions} and Theorem~\ref{thm:perturb}.

		\textbf{Step 2}: Specializing to~\eqref{eq:sto_heat}.

		Choose $\theta = -\frac{\log(r)}{\log(R/r)} \in (0,1)$. Then $F(\theta) = 1$ and we find $A_\theta = A$ and $B_{\theta} = B$. Write $p_0 \coloneqq q_\theta$, then Step~1 shows
		\begin{align}
			\| u \|_{\V^{p_0}} \leq C \bigl( \| f \|_{\E^{p_0}} + \| g \|_{\mathbb{E}^{p_0}_\frac{1}{2}} \bigr), \quad f\in \Eq, g\in \Eqgam,
		\end{align}
		where $2 < p_0 < q$ and $u$ is the unique solution to~\eqref{eq:sto_heat} with data $(f,g)$. By an approximation argument, we extend well-posedness to $(f,g) \in \E^{p_0} \times \mathbb{E}^{p_0}_\frac{1}{2}$.

Now to obtain the statement for $p\in (2, p_0]$, one can either use complex interpolation with the case $p=2$, or lower the value of $r$ in the above proof.
	\end{proof}

	Now we can prove the main regularity result for linear equations stated in Theorem~\ref{thm:main1}.

	\begin{proof}[Proof of Theorem~\ref{thm:main1}]
		By considering $\e^{\lambda t} u(t)$ for suitable $\lambda\in \R$ we can reduce to the case $M=0$ in Assumption~\ref{ass:SP}. Then Proposition~\ref{prop:main} yields a unique strong solution to~\eqref{eq:sto_heat} with $u_0 = 0$ that satisfies the $\L^p$-estimate corresponding to $\theta = 0$ in the theorem. Still with $u_0 = 0$, we can first deduce higher regularity in time from~\cite[Prop.~3.8]{AVloc}, using the reference operators $\wt{B} = 0$ and $\wt{A}=A_0$. In a second step, we include initial values $u_0 \in \L^p_{\cF_0}(\Omega; (H,V)_{1-\nicefrac{2}{p},p})$ in virtue of~\cite[Prop.~3.10]{AVloc}. At both stages,~\cite[Prop.~2.10]{AVloc} can be used to get the maximal estimate missing in the case $\theta = 0$.
	\end{proof}

	\section{Quasilinear problem and well-posedness result}
	\label{sec:main_result2}

	On an open and bounded set $\Dom\subseteq \R^d$ consider the quasilinear system
	\begin{align}
		\label{eq:ql}
		\tag{QL}
		\left\{\quad
		\begin{aligned}
			\d u^\alpha  &= \big[ \partial_i (a^{\alpha\beta}_{ij}(u) \partial_j u^\beta) + \partial_i \Phi^\alpha_i(u) +\phi^\alpha(u) \big] \d t
			\\  & \qquad \qquad
			+ \sum_{n\geq 1}  \big[ b^{\alpha\beta}_{n,j}(u^\beta) \partial_j u^\beta + g^\alpha_{n}(u) \big] \d w_n,\ \ &\text{ on }\Dom,\\
			u^\alpha &= 0, \ \ &\text{ on }\partial\Dom,
			\\
			u^\alpha(0)&=u^\alpha_{0},\ \ &\text{ on }\Dom,
		\end{aligned}\right.
	\end{align}
	where $W = (w_n(t)\,:\,t\geq 0)_{n\geq 1}$ are independent standard Brownian motions on a probability space $\Omega$, $i,j = 1, \dots, d$, and $\alpha,\beta = 1, \dots, N$ for some system size $N \geq 1$. For simplicity, we employed Einstein's summation convention above and in the sequel. We will ignore the system size $N$ in the notation for function spaces if this cannot cause any confusion. Let $\cF$ be the filtration on $\Omega$ induced by $W$.

	Let $h > 1$. For our main result with growing non-linearities, Theorem~\ref{thm:quasi2}, we introduce the following assumption. As an intermediate step, we also treat \enquote{the case $h=1$} in Theorem~\ref{thm:quasi1}, with reduced structural assumptions, and thus of independent interest.

	\begin{ass}[$h$]\label{ass:ql}
		\begin{enumerate}[{\rm(1)}]
			\item \label{ass:2ndorder10} The system is diagonal, that is $a_{ij}^{\alpha\beta}$ and $b_j^{\alpha\beta}$ are non-zero only when $\alpha = \beta$. Write $a_{ij}^{\alpha} \coloneqq a_{ij}^{\alpha\alpha}$ and $b_j^{\alpha} \coloneqq b_j^{\alpha\alpha}$. For fixed $\alpha$, the matrix $a^\alpha$ does not need to be diagonal.
			\item\label{ass:2ndorder11} The coefficients $a^{\alpha}_{ij}\colon (0,T) \times \Dom \times \R^N \to \R$ and $b^{\alpha}_j \coloneqq (b^{\alpha}_{n,j})_{n\geq 1} \colon (0,T) \times \Dom \times \R\to \ell^2$ are continuous in the last component, measurable, and uniformly bounded by a constant $\Lambda$. Moreover, $\sup_{t,x,y} |\partial_j b_{n,j}^{\alpha\beta}(t,x,y)| < \infty$.
			\item\label{ass:2ndorder12}
			One has $a^{\alpha}_{ij} = a^{\alpha}_{ji}$ and there exists $\lambda>0$ such that, for all $t\in (0,T)$, $x\in \Dom$, and $y\in \R^N$,
			$$
			\big(a^{\alpha}_{ij}(t,x,y)-\frac12 b^{\alpha}_{n,i}(t,x,y^\alpha) b^{\alpha}_{n,j}(t,x,y^\alpha) \big)
			\xi_i^\alpha \xi_j^\alpha
			\geq  \lambda |\xi|^2, \quad \xi\in \R^{dN}.
			$$
			\item\label{ass:2ndorder13} The non-linearities $\Phi^\alpha_i \colon (0,T) \times \Dom \times \R^N \to \R$ and $\phi^\alpha \colon (0,T) \times \Dom \times \R^N \to \R$ are measurable, locally Lipschitz in the last component, and satisfy the following growth condition: there exists a constant $C$ such that,
			\begin{align}\label{eq:condPhiphi} \tag{G}
				|\Phi(t,x,y)| + |\phi(t,x,y)|&\leq C(1+|y|^{h}), \quad  t\in (0,T), x\in \Dom,  y\in \R^N.
			\end{align}
			\item\label{ass:2ndorder14} The non-linearity $\phi$ is dissipative in the following sense: there exists a constant $C$ such that one has for all $t\in (0,T)$, $x\in \Dom$, $y\in \R^N$ and $\alpha=1,\dots,N$ that
			\begin{align}\label{eq:phi_dissip}
				\phi^\alpha(t,x,y) y^\alpha \leq C(|y|^2+1).
			\end{align}
The nonlinearity $\Phi$ is dissipative in the following sense: \[\Phi^{\alpha}(t,x,y) = \overline{\Phi}^{\alpha}(y^{\alpha}) + \wh{\Phi}^\alpha(t,x,y),\]
where $\overline{\Phi}_i:\R\to \R^N$ is measurable, and $\wh{\Phi}_i:(0,T) \times \Dom \times \R^N\to \R^N$ is measurable and Lipschitz continuous in the last component (uniformly with respect to $t$ and $x$). 
			\item\label{ass:2ndorder15} The stochastic non-linearity $g^\alpha \coloneqq (g^\alpha_n)_{n \geq 1} \colon (0,T) \times \Dom \times \R^N \to \ell^2$ is measurable, Lipschitz continuous and of linear growth in the last component (uniformly with respect to the first two components).
		\end{enumerate}
	\end{ass}

	\begin{remark}
		Some discussion of the hypotheses is in order.
		\begin{enumerate}
			\item When we work with non-linearities of linear growth, see Section~\ref{sec:ql_lip}, then the diagonal structure assumed in~\ref{ass:2ndorder10} is not needed. This is why~\eqref{eq:ql} is formulated for the non-diagonal case.
			\item The non-linearities for $a$, $\Phi$ and $\phi$ lead to a coupling between the components even when working with a diagonal structure.
			\item To identify the limit in the stochastic compactness method, it seems crucial that $b_j^{\alpha\beta}$ only depends on the $\beta$th component. In fact, it is a surprising strength of our approach that this structural assumption is not needed for $a$.
		\end{enumerate}
	\end{remark}

	\begin{remark}[Diagonal structure]
		\label{rem:diagonal}
		Instead of imposing a diagonal structure, we could also strengthen the coercivity condition as follows: for $q\in [0,2(h-1)+\eps]$ assume that there exists $\lambda_q>0$ such that, for all $t\in (0,T)$, $x\in \Dom$, and $y\in \R^N$,
		$$
		|y^\alpha|^q \Big(a^{\alpha}_{ij}(t,x,y)-\frac12 b^{\alpha}_{n,i}(t,x,y^\alpha) b^{\alpha}_{n,j}(t,x,y^\alpha) \Big)
		\xi_i^\alpha \xi_j^\alpha
		\geq  \lambda_q |y^\alpha|^q|\xi^\alpha|^2 \quad \text{ for all }\xi\in \R^{dN}.
		$$
		For instance, such a condition can be ensured for perturbations of diagonal systems in which $|a^{\alpha\beta}_{ij}| \leq c (1+|y|^q)^{-1}$ whenever $\alpha \neq \beta$, where $c$ is sufficiently small. For clarity of exposition, we only work with the diagonal case. This case is also the most important one in reaction--diffusion equations. However, we do not know whether the structural condition we have made is essential. It is needed at a technical point in the proof in the identification of the limit in the stochastic compactness method.
	\end{remark}

	To make precise the homogeneous Dirichlet boundary condition in~\eqref{eq:ql}, we introduce the following function spaces.
	\begin{definition}
		Write $\H^1_0(\Dom)$ for the closure in $\H^1(\Dom)$ of $\Cont^\infty_0(\Dom)$, the smooth and compactly supported functions in $\Dom$. Moreover, for $s\in (0,1)$ put
		\begin{align}
			\H^s_0(\Dom) \coloneqq [\L^2(\Dom), \H^1_0(\Dom)]_s \quad \& \quad \B^s_{2,p,0}(\Dom) \coloneqq (\L^2(\Dom), \H^1_0(\Dom))_{s,p}.
		\end{align}
	\end{definition}

	\begin{remark}
		The spaces $\H^s_0(\Dom)$ and $\B^s_{2,p,0}(\Dom)$ can often be identified with closed subspaces of the usual Bessel potential and Besov spaces on $\Dom$.
	\end{remark}

	The main result on the quasi-linear system \eqref{eq:ql} reads as follows.

	\begin{theorem}[Weak existence]\label{thm:quasi2}
		Let $h > 1$. Suppose Assumption \ref{ass:ql} holds. Then there is $p_0 > 2$, depending only on $C, \Lambda, \lambda$ from the assumption as well as dimensions, such that, if $p\in (2, p_0]$ and $q\in (ph, \infty)$, then, given $u_0\in \L^{p}_{\cF_0}(\Omega;\B^{1-\nicefrac2p}_{2,p,0}(\Dom))\cap \L^q_{\cF_0}(\Omega \times \Dom)$, there exists a weak solution $(\wt{u}, \wt{W}, \wt{\Omega}, \wt{\cF}, \wt{\P}, (\wt{\cF}_t)_{t\geq 0})$ of \eqref{eq:ql} in the sense of Definition~\ref{def:weak_sol}. One has the estimates
		\begin{align}
			\label{eq:quasi2a} \| \wt{u} \|_{\L^p(\wt{\Omega}; \Cont([0,T]; \B^{1-\nicefrac{2}{p}}_{2,p,0}(\Dom)))}&\lesssim_\theta 1+\|u_0\|_{\L^{q}(\Omega;\L^{q}(\Dom))}^h + \|u_0\|_{\L^{p}(\Omega;\B^{1-\nicefrac2p}_{2,p,0}(\Dom))}, \\
			\label{eq:quasi2b} \|\wt{u} \|_{\L^q(\wt{\Omega} \times (0,T) \times \Dom)} &\lesssim 1 + \|u_0\|_{\L^{q}(\Omega \times \Dom)}^h.
		\intertext{Moreover, for all $\theta\in [0,1/2)$ one has the bound}
			\label{eq:quasi2c}\|\wt{u}\|_{\L^{p}(\wt{\Omega}; \H^{\theta,p}(0,T;\H^{1-2\theta}_0(\Dom)))}&\lesssim_\theta 1+\|u_0\|_{\L^{q}(\Omega;\L^{q}(\Dom))}^h + \|u_0\|_{\L^{p}(\Omega;\B^{1-\nicefrac2p}_{2,p,0}(\Dom))}.
		\end{align}
	\end{theorem}

	Our notion of a weak solution employed in Theorem~\ref{thm:quasi2} is the following.

	\begin{definition}\label{def:weak_sol}
		Call the tuple $(\wt{u}, \wt{W}, \wt{\Omega}, \wt{\cF}, \wt{\P}, (\wt{\cF}_t)_{t\geq 0})$ a \emph{weak solution} of \eqref{eq:ql} if  $\wt{W}=(\wt{w}_n)_{n\geq 1}$ is an $\ell^2$-cylindrical Brownian motion on $\wt{\Omega}$, $\wt{u}(0)$ has the same distribution as $u_0$,
		$\wt{u}\in \L^2(0,T;\H^1_0(\Dom)) \cap \Cont([0,T];\L^2(\Dom))$ almost surely and is $(\wt{\cF}_t)$-progressively measurable,
		$\Phi^\alpha_i(\wt{u}), \phi^\alpha(\wt{u})\in \L^1(0,T;\L^1(\Dom))$, $g(\wt{u}) \in \L^2(0,T;\L^2(\Dom;\ell^2))$, and for all $\xi\in \Cont^\infty_0(\Dom)$ and $t\in [0,T]$, one has almost surely
		\begin{align}
			\lb \wt{u}^\alpha(t), \xi\rb  = \lb \wt{u}^\alpha(0),\xi\rb &+ \int_0^t -\lb a^{\alpha\beta}_{ij}(\wt{u}) \partial_j \wt{u}, \partial_i \xi\rb  - \lb \Phi^\alpha_i(\wt{u}), \partial_i \xi\rb + \lb \phi^\alpha(\wt{u}), \xi\rb \d t
			\\
			&+ \sum_{n\geq 1}  \int_0^t \lb b^{\alpha\beta}_{n,j}(\wt{u}) \partial_j \wt{u}, \xi\rb + \lb g^\alpha_{n}(\wt{u}),\xi\rb \d w_n.
		\end{align}
	\end{definition}

	The proof of Theorem~\ref{thm:quasi2} will occupy the rest of this article.

\begin{remark}
In the scalar case $N=1$ with $b = 0$ and $\phi=0$ a result such as Theorem \ref{thm:quasi2} is proved in \cite{DHV} in the case of periodic boundary conditions. Moreover, even the degenerate case and initial data in $\L^1$ are considered via a different solution concept called kinetic solutions. In this framework uniqueness is proved as well. It would be interesting to see if some of these results can be obtained in the generality of Theorem \ref{thm:quasi2}.

In the scalar case $N=1$ one can also make a comparison to \cite[Example 4.1]{rockner2022well} in case the parameter $\alpha$ used there satisfies $\alpha=2$. The main differences in this special case are that, in Theorem  \ref{thm:quasi2}, we allow systems. Moreover, our condition on the gradient noise is more flexible, and for $N=1$ it coincides with the classical stochastic parabolicity condition (cf.\ \cite[Assumption 5.9(2)]{AVvar}). To compare the condition on the gradient noise more precisely, note that the condition in \cite[Theorem 3.2]{rockner2022well} reduces to $\chi L_B<2L_A$, where $\chi\geq h+1$, where $h\geq 1$ is the growth of our nonlinearity. Moreover, $L_A$ and $L_B$  are uniform constants over the coefficients. Our condition in Assumption \ref{ass:ql}\ref{ass:2ndorder12} is pointwise in $x$ and $y$ instead of uniform. Moreover, the growth of $\phi$ and $\Phi$ does not influence this coercivity condition in any way. 
\end{remark}

	\section{Quasilinear problem with Lipschitz non-linearities}
	\label{sec:ql_lip}

	First, we study~\eqref{eq:ql} under the following set of hypotheses. In Section~\ref{sec:ql_growing} we are going to reduce the general case to the current one in virtue of an approximation procedure.

	\begin{ass}\label{ass:ql_lip}
		\begin{enumerate}[{\rm(1)}]
			\item\label{ass:2ndorder21} The coefficient functions $a^{\alpha\beta}_{ij}\colon (0,T) \times \Dom \times \R^N \to \R$ and $b^{\alpha\beta}_j \coloneqq (b^{\alpha\beta}_{n,j})_{n\geq 1} \colon (0,T) \times \Dom \times \R\to \ell^2$ are continuous in the last component, measurable, and uniformly bounded by a constant $\Lambda$.
Moreover, $\sup_{t,x,y} |\partial_j b_{n,j}^{\alpha\beta}(t,x,y)| < \infty$. \item\label{ass:2ndorder22}
			One has $a^{\alpha\beta}_{ij} = a^{\beta\alpha}_{ji}$ and there exists $\lambda>0$ such that, for all $t\in (0,T)$, $x\in \Dom$, and $y\in \R^N$,
			$$
			\big(a^{\alpha\beta}_{ij}(t,x,y)-\frac12 b^{\gamma\alpha}_{n,i}(t,x,y^\alpha) b^{\gamma\beta}_{n,j}(t,x,y^\beta) \big)
			\xi_i^\alpha \xi_j^\beta
			\geq  \lambda |\xi|^2 \quad \text{ for all }\xi\in \R^{dN}.
			$$
			\item\label{ass:2ndorder23} The non-linearities $\Phi^\alpha_i \colon (0,T) \times \Dom \times \R^N \to \R$ and $\phi^\alpha \colon (0,T) \times \Dom \times \R^N \to \R$, and the stochastic non-linearity $g^\alpha \coloneqq (g^\alpha_n)_{n \geq 1} \colon (0,T) \times \Dom \times \R^N \to \ell^2$ are measurable, Lipschitz continuous and of linear growth in the last component (uniformly with respect to the first two components).
		\end{enumerate}
	\end{ass}

	Observe that Assumption~\ref{ass:ql_lip}~\ref{ass:2ndorder23} is the combination of Assumption~\ref{ass:ql}~\ref{ass:2ndorder13} with $h=1$ and Assumption~\ref{ass:ql}~\ref{ass:2ndorder15}. Assumption~\ref{ass:ql}~\ref{ass:2ndorder14} becomes obsolete.

	We are going to show the following well-posedness result.

	\begin{theorem}
		\label{thm:quasi1}
		Suppose Assumption \ref{ass:ql_lip} holds. There is $p_0 > 2$, depending only on $C, \Lambda, \lambda$ from the assumption as well as dimensions, such that if $p\in (2, p_0]$, then, given $u_0\in \L^{p}_{\cF_0}(\Omega;\B^{1-\nicefrac2p}_{2,p,0}(\Dom))$, there exists a weak solution $(\wt{u}, \wt{W}, \wt{\Omega}, \wt{\cF}, \wt{\P}, (\wt{\cF}_t)_{t\geq 0})$ of \eqref{eq:ql}. Moreover, for all $\theta\in [0,1/2)$ one has
		\begin{align}
			\label{eq:quasi1a}
			\| \wt{u} \|_{\L^p(\wt{\Omega}; \Cont([0,T]; \B^{1-\nicefrac{2}{p}}_{2,p,0}(\Dom)))} + \|\wt{u}\|_{\L^{p}(\wt{\Omega};\H^{\theta,p}(0,T;\H^{1-2\theta}_0(\Dom)))}\lesssim_\theta 1+ \|u_0\|_{\L^{p}(\Omega;\B^{1-\nicefrac2p}_{2,p,0}(\Dom))}.
		\end{align}
	\end{theorem}

	\subsection{Smoothing of the coefficients}
	\label{subsec:smoothing}

	Let $\zeta \colon \R \to [0,\infty)$ be a smooth and compactly supported function with integral equal to $1$ and $\rho \colon \R^N \to [0,\infty)$ the $N$-fold product of $\zeta$. For $m \geq 1$, write $\zeta_m$ and $\rho_m$ for the induced mollifier sequences.
	Now, define the quasi-linear operators
	\begin{align}
		{}_m A(t) \colon \L^2(\Dom)^N &\to \cL(\H^1_0(\Dom)^N, \H^{-1}(\Dom)^N), \\
		{}_m B(t) \colon \L^2(\Dom) &\to \cL(\H^1_0(\Dom)^N,\cL_2(\ell^2, \L^2(\Dom)^N))
	\end{align}
	by
	\begin{align}
		{}_m A(t,v) u &= \bigl( \partial_i ({}_m a_{ij}^{\alpha\beta}(t,x,R_m v) \partial_j u^\beta) \bigr)_{\alpha=1}^N, \qquad (u\in \H^1_0(\Dom)^N, v\in \L^2(\Dom)^N), \\
		\lb {}_m B(t,w)u, e_n \rb &= \bigl( {}_m b_{n,j}^{\alpha\beta}(t,x,R_m w^\beta) \partial_j u^\beta \bigr)_{\alpha=1}^N, \qquad (u\in \H^1_0(\Dom)^N, w\in \L^2(\Dom)),
	\end{align}
	where $R_m f = \rho_m \ast E f$, ${}_m a_{ij}^{\alpha\beta}(t,x,\cdot) = \rho_m \ast a_{ij}^{\alpha\beta}(t,x,\cdot)$ and ${}_m b_{n,j}^{\alpha\beta}(t,x,\cdot) = \zeta_m \ast b_{n,j}^{\alpha\beta}(t,x,\cdot)$, where the convolution is taken with respect to the third variable. Here, $E$ denotes the zero extension of a given function, and we apply the convolution in the definition of $R_m$ component-wise if $f$ is vector-valued.

	Consider now the family of equations

	\begin{align}
		\label{eq:ql_m}
		\tag{m-QL}
		\left\{\quad
		\begin{aligned}
			\d u  &= \big[ {}_m A(u)u + \partial_i \Phi_i(u) +\phi(u) \big] \d t
			\\  & \qquad \qquad
			+ \sum_{n\geq 1}  \big[ \lb {}_m B(u)u, e_n \rb + g_{n}(u) \big] \d w_n,\ \ &\text{ on }\Dom,\\
			u &= 0, \ \ &\text{ on }\partial\Dom,
			\\
			u(0)&=u_{0},\ \ &\text{ on }\Dom.
		\end{aligned}\right.
	\end{align}

	Following~\cite[p.~211]{KryAnalytic} we show that the smoothed coefficients satisfy the same coercivity condition as the original ones.

	\begin{lemma}
		\label{lem:smooth_elliptic}
		For fixed $m$, the smoothed coefficients ${}_m a_{ij}^{\alpha\beta}$ and ${}_m b_{n,j}^{\alpha\beta}$ satisfy Assumptions~\ref{ass:ql_lip}~\ref{ass:2ndorder21} and~~\ref{ass:2ndorder22} with the same constants.
	\end{lemma}

	\begin{proof}
		Fix $m$. By definition of the convolution, symmetry of ${}_m a^{\alpha\beta}_{ij}$ is immediate. The same holds for uniform boundedness, since the convolution kernel has normalized $\L^1$-norm.

		For better readability, suppress dependence on $t$ and $x$ of the coefficients. Introduce the vector notation $\vec{b}_{n,j}^\alpha(y) \coloneqq \bigl(b_{n,j}^{\alpha\beta}(y^\beta)\bigr)_{\beta=1}^N$, likewise for ${}_m \vec{b}_{n,j}^\alpha(y)$. In this way, $\sum_\beta b^{\alpha\beta}_{n,j}(y^\beta) \xi_j^\beta = \vec{b}_{n,j}^\alpha(y) \cdot \xi_j$ and $({}_m b_{n,j}^{\alpha\beta})_{\beta=1}^N = \rho_m \ast \vec{b}_{n,j}^\alpha$, where again we only convolve in the last coordinate.

		Observe first that for $y\in \R^N$,
		\begin{align}
			\sum_n \sum_{i,j} \sum_{\alpha,\beta,\gamma} {}_m b^{\gamma\alpha}_{n,i}(y^\alpha) {}_m b^{\gamma\beta}_{n,j}(y^\beta) \xi_i^\alpha \xi_j^\beta = \sum_n \sum_\gamma \Bigl| \sum_j {}_m \vec{b}^{\gamma}_{j}(y) \cdot \xi_j \Bigr|_2^2,
		\end{align}
		where $|\cdot|_2$ is the Euclidean norm on $\R^N$.
		Write $\boldsymbol{z} \colon \R^N \ni z \mapsto z$ for the identity map.
		Using that the convolution kernel is positive with normalized integral, Jensen's inequality gives
		\begin{align}
			\sum_n \sum_\gamma \Bigl| \sum_j {}_m \vec{b}^{\gamma}_{j}(y) \cdot \xi_j \Bigr|_2^2
			&\leq \sum_n \sum_\gamma \Bigl| \sum_j \vec{b}^{\gamma}_{j}(\boldsymbol{z}) \cdot \xi_j \Bigr|_2^2 \ast \rho_m(y) \\
			&= \sum_n \sum_{i,j} \sum_{\alpha,\beta,\gamma} \Bigl( b^{\gamma\alpha}_{n,i}(\boldsymbol{z}^\alpha) b^{\gamma\beta}_{n,j}(\boldsymbol{z}^\beta) \xi_i^\alpha \xi_j^\beta \Bigr) \ast \rho_m(y).
		\end{align}
		Hence, it follows from Assumptions~\ref{ass:ql_lip}~\ref{ass:2ndorder22} and positivity of the convolution kernel that
		\begin{align}
			& \quad \Big({}_m a^{\alpha\beta}_{ij}(y)-\frac12 {}_m b^{\gamma\alpha}_{n,i}(y^\alpha) {}_m b^{\gamma\beta}_{n,j}(y^\beta) \Big) \xi_i^\alpha \xi_j^\beta  \\
			& \geq \int_{\R^N} \Big(a^{\alpha\beta}_{ij}(z)-\frac12 b^{\gamma\alpha}_{n,i}(z^\alpha) b^{\gamma\beta}_{n,j}(z^\beta) \Big) \xi_i^\alpha \xi_j^\beta \rho_m(y-z) \d z
			\geq \lambda | \xi |^2. \qedhere
		\end{align}
	\end{proof}

Using \cite{AVvar} we show that for each $m$, problem~\eqref{eq:ql_m} admits a unique solution.

	\begin{lemma}
		\label{lem:ql-reg}
		Suppose Assumption~\ref{ass:ql_lip} holds. Then, for each $m$, there is a unique strong solution $u_m \in \L^2(\Omega \times (0,T); \H^1_0(\Dom)) \cap \L^2(\Omega; \Cont([0,T]; \L^2(\Dom)))$ to~\eqref{eq:ql_m}.
	\end{lemma}
	\begin{proof}
		Fix $m$. Write $\|\cdot\|_{\HS} = \|\cdot\|_{\HS(\ell^2, \L^2(\Dom))}$. The assertion follows from~\cite[Thm.~3.5]{AVvar} applied with $A_0(v)u \coloneqq A_0(t,v)u \coloneqq {}_m A(t,v)u$ and $B_0(v)u \coloneqq B_0(t,v)u \coloneqq {}_m B(t,v)u$, the non-linearities $F(v) \coloneqq F(t,v) \coloneqq \partial_i \Phi_i(t,v) + \phi(t,v)$ and $G(v)e_n \coloneqq G(t,v)e_n \coloneqq g_{n}(t,v)$, and the cylindrical Brownian motion $We_n \coloneqq w_n$.

		Indeed, note first that, in the light of Lemma~\ref{lem:smooth_elliptic}, all assertions from Assumption~\ref{ass:ql_lip} remain valid for the problem~\eqref{eq:ql_m}. Besides boundedness of $A_0$ and $B_0$, and linear growth of $F$ and $G$ in the respective norms, which readily follow from uniform boundedness of the coefficients in the first two cases, and linear growth of the non-linearities in the last two cases, we have to check the regularity conditions
		\begin{align}
			\label{eq:AVvar_reg1}
			\| A_0(u)w - A_0(v)w \|_{\H^{-1}(\Dom)} + \| B_0(u)w - B_0(v)w \|_{\HS} \lesssim \| u - v \|_{\L^2(\Dom)} \| w \|_{\H^1_0(\Dom)},
		\end{align}
		for all $u,v\in \L^2(D)$ and $w \in \H^1_0(\Dom)$, and
		\begin{align}
			\label{eq:AVvar_reg2}
			\| F(u) - F(v) \|_{\H^{-1}(\Dom)} + \| G(u) - G(v) \|_{\HS} \lesssim \| u - v \|_{\L^2(\Dom)},
		\end{align}
		for all $u,v\in \L^2(D)$,
		as well as the following coercivity condition: there exists $\theta, \eta > 0$ and $M \geq 0$ such that a.s.
		\begin{align}
			\label{eq:AVvar_coerc}
			\lb A_0(v)u, u \rb - (\nicefrac{1}{2} + \eta) \| B_0(v)u \|_{\HS}^2 \geq \theta \| u \|_{\H^1_0(\Dom)}^2 - M \| u \|_{\L^2(\Dom)}^2,
		\end{align}
		for all $u\in \H^1_0(\Dom)$ and $v\in \L^2(\Dom)$.

		The second regularity condition,~\eqref{eq:AVvar_reg2}, follows immediately from Lischitz continuity of $\Phi$, $\phi$, and $g$. For brevity, we present the regularity condition~\eqref{eq:AVvar_reg1} only for $A_0$. The calculation for $B_0$ is similar. Fix an integer $k>\nicefrac{d}{2}$. Using the Sobolev embedding and Young's convolution inequality, calculate
		\begin{align}
			\| A_0(u)w - A_0(v)w \|_{\H^{-1}(\Dom)} &\leq \|{}_m a^{\alpha\beta}_{ij}(R_m u) - {}_m a^{\alpha\beta}_{ij}(R_m v) \|_{\infty} \| \partial_j w^\beta \|_{\L^2(\Dom)}
			\\ &\lesssim \| {}_m a^{\alpha\beta}_{ij} \|_{\Lip} \| R_m u - R_m v \|_{\infty} \| \partial_j w^\beta \|_{\L^2(\Dom)}
			\\ &\lesssim_m \| a^{\alpha\beta}_{ij} \|_{\infty} \| R_m u - R_m v \|_{\W^{k,2}(\Dom)} \| \partial_j w^\beta \|_{\L^2(\Dom)}
			\\ &\lesssim \| \rho_m \|_{\W^{k,1}(\Dom)} \| u - v \|_{\L^2(\Dom)} \| w \|_{\H^1_0(\Dom)}
			\\ &\lesssim_m \| u - v \|_{\L^2(\Dom)} \| w \|_{\H^1_0(\Dom)}.
		\end{align}

		To complete the proof, note first that Assumption~\ref{ass:ql_lip}~\ref{ass:2ndorder22} implies~\eqref{eq:AVvar_coerc} with $\eta = 0$ and $\theta \coloneqq \lambda \eqqcolon M$. Therefore, since $B_0(v) \colon \H^1_0(\Dom)^N \to \L^2(\Dom)^N$ is bounded (uniformly in $v$), we get~\eqref{eq:AVvar_coerc} with $\eta > 0$ if we replace $\theta$ by $\nicefrac{\theta}{2}$.
	\end{proof}

	\subsection{Higher integrability of semilinear equations}
	\label{subsec:semi}

	It will turn out useful in the sequel to consider a semi-linear version of~\eqref{eq:ql}. More precisely, consider
	\begin{align}
		\label{eq:sl}
		\tag{SL}
		\left\{\quad
		\begin{aligned}
			\d u^\alpha  &= \big[ \partial_i (\boldsymbol{a}^{\alpha\beta}_{ij} \partial_j u^\beta) + \partial_i \Phi^\alpha_i(u) +\phi^\alpha(u) \big] \d t
			\\  & \qquad \qquad
			+ \sum_{n\geq 1}  \big[ \boldsymbol{b}^{\alpha\beta}_{n,j} \partial_j u^\beta + g^\alpha_{n}(u) \big] \d w_n,\ \ &\text{ on }\Dom,\\
			u^\alpha &= 0, \ \ &\text{ on }\partial\Dom,
			\\
			u^\alpha(0)&=u^\alpha_{0},\ \ &\text{ on }\Dom.
		\end{aligned}\right.
	\end{align}
	We impose the following assumption, which is similar to Assumption~\ref{ass:ql_lip}.
	\begin{ass}\label{ass:secondorder-semi}
		\begin{enumerate}[{\rm(1)}]
			\item\label{ass:2ndorder31} The coefficients $\boldsymbol{a}^{\alpha\beta}_{ij}\colon \Omega \times (0,T) \times \Dom \to \R$ as well as $\boldsymbol{b}^{\alpha\beta}_j \coloneqq (\boldsymbol{b}^{\alpha\beta}_{n,j})_{n\geq 1} \colon \Omega \times (0,T) \times \Dom \to \ell^2$ are progressively measurable and uniformly bounded by a constant $\Lambda$.
			\item\label{ass:2ndorder32}
			One has $\boldsymbol{a}^{\alpha\beta}_{ij} = \boldsymbol{a}^{\beta\alpha}_{ji}$ and there exists $\lambda>0$ such that, for all $t\in (0,T), x\in \Dom$, one has almost surely
			$$
			\Big(\boldsymbol{a}^{\alpha\beta}_{ij}(t,x)-\frac12 \boldsymbol{b}^{\gamma\alpha}_{n,i}(t,x) \boldsymbol{b}^{\gamma\beta}_{n,j}(t,x) \Big)
			\xi_i^\alpha \xi_j^\beta
			\geq  \lambda |\xi|^2 \quad \text{ for all }\xi\in \R^{dN}.
			$$
			\item\label{ass:2ndorder33} The deterministic non-linearities $\Phi^\alpha_i \colon (0,T) \times \Dom \times \R^N \to \R$ and $\phi^\alpha \colon (0,T) \times \Dom \times \R^N \to \R$ and the stochastic non-linearity $g^\alpha \coloneqq (g^\alpha_n)_{n \geq 1} \colon (0,T) \times \Dom \times \R^N \to \ell^2$ are measurable and Lipschitz continuous and of linear growth in the last variable (uniformly with respect to the first two components).
		\end{enumerate}
	\end{ass}
	A prototypical example for Assumption~\ref{ass:secondorder-semi} is $\boldsymbol{a}^{\alpha\beta}_{ij}(t,x) \coloneqq a^{\alpha\beta}_{ij}(t,x,u(t,x))$ and $\boldsymbol{b}^{\alpha\beta}_{n,j}(t,x) \coloneqq b^{\alpha\beta}_{n,j}(t,x,u^\beta(t,x))$, where $a^{\alpha\beta}_{ij}$ and $b^{\alpha\beta}_{n,j}$ are coefficients for~\eqref{eq:ql} subject to Assumption~\ref{ass:ql_lip} and the progressively measurable function $u \colon \Omega \times (0,T) \times \Dom \to \R^N$ is frozen.

	It is well-known that~\eqref{eq:sl} admits a unique variational solution in $\L^2(\Omega \times (0,T); \H^1_0(\Dom)) \cap \L^2(\Omega; \Cont([0,T]; \L^2(\Dom)))$. As a consequence of Theorem~\ref{thm:main1}, the main result of the first part of this article, we can improve integrability in time of this variational solution to $p>2$.

	\begin{proposition}[Higher integrability for~\eqref{eq:sl}]
		\label{prop:sl_higher_int}
		Suppose Assumption~\ref{ass:secondorder-semi} holds. Then there is some $p_0 > 2$ such that for any $p\in (2,p_0]$ and any $u_0 \in \L^p_{\cF_0}(\Omega; \B^{1-\nicefrac{2}{p}}_{2,p,0}(\Dom))$ there is a unique solution $u \in \L^p(\Omega \times (0,T); \H^1_0(\Dom)) \cap \L^p(\Omega; \Cont([0,T]; \B^{1-\nicefrac{2}{p}}_{2,p,0}(\Dom)))$ to~\eqref{eq:sl} satisfying the estimate
		\begin{align}
			\label{eq:sl_bound1}
			\| u \|_{\L^p(\Omega; \Cont([0,T]; \B^{1-\nicefrac{2}{p}}_{2,p,0}(\Dom)))} \lesssim 1 + \| u_0 \|_{\L^p(\Omega; \B^{1-\nicefrac{2}{p}}_{2,p,0}(\Dom))},
		\end{align}
		and for all $\theta \in [0,\nicefrac{1}{2})$ the estimate
		\begin{align}
			\label{eq:sl_bound2}
			\| u \|_{\L^p(\Omega; \H^{\theta,p}(0,T; \H^{1-2\theta}_0(\Dom)))} \lesssim_\theta 1 + \| u_0 \|_{\L^p(\Omega; \B^{1-\nicefrac{2}{p}}_{2,p,0}(\Dom))}.
		\end{align}
		If $\theta \in (\nicefrac{1}{p}, \nicefrac{1}{2})$, then one has moreover the estimate
		\begin{align}
			\label{eq:sl_bound3}
			\| u \|_{\L^p(\Omega; \Cont^{\theta-\nicefrac{1}{p}}([0,T]; \H^{1-2\theta}_0(\Dom)))} \lesssim_\theta 1 + \| u_0 \|_{\L^p(\Omega; \B^{1-\nicefrac{2}{p}}_{2,p,0}(\Dom))}.
		\end{align}
	\end{proposition}

	\begin{proof}
		\textbf{Step~1}: Existence of variational solution and bootstrapping integrability.

		Consider the progressively measurable maps
		\begin{align}
			A_0(t,u) &\coloneqq (\partial_i (\boldsymbol{a}^{\alpha\beta}_{ij}(t, x) \partial_j u^\beta))_{\alpha=1}^N, \\
			B_0(t,u) e_n &\coloneqq (\boldsymbol{b}^{\alpha\beta}_{n,j}(t, x) \partial_j u^\beta)_{\alpha=1}^N, \\
			F(u) &\coloneqq (\partial_i \Phi^\alpha_i(u) +\phi^\alpha(u))_{\alpha=1}^N,\\
			G(u) e_n &\coloneqq (g^\alpha_{n}(u))_{\alpha=1}^N.
		\end{align}
		The well-posedness result~\cite[Thm.~2.4]{GHV} applied with $A(t,v) \coloneqq A_0(t,v) + F(v)$ and $B(t,v)e_n \coloneqq B_0(t,v) e_n + G(u) e_n$ yields a unique strong solution $$u \in \L^p(\Omega; \L^2(0,T; \H^1_0(\Dom))) \cap \L^p(\Omega; \Cont([0,T]; \L^2(\Dom)))$$ together with the estimate
		\begin{align}
			\label{eq:sl-var-bound}
			\| u \|_{\L^p(\Omega; \L^2(0,T; \H^1_0(\Dom)))} + \| u \|_{\L^p(\Omega; \Cont([0,T]; \L^2(\Dom)))} \lesssim 1 + \| u_0 \|_{\L^p(\Omega; \L^2(\Dom))}
		\end{align}
		for some $p > 2$ depending only on the quantities from Assumption~\ref{ass:secondorder-semi}, provided we can verify the hypotheses listed in~\cite[Ass.~2.1]{GHV}. Using boundedness of the coefficients and linear growth of the non-linearities, (H4) and (H5) with $f \equiv 1$ follow readily. As in the proof of Lemma~\ref{lem:ql-reg} one deduces (H3) with $p > 2$ sufficiently close to $2$ from Assumption~\ref{ass:secondorder-semi}~\ref{ass:2ndorder32} and boundedness of $B_0$. Condition (H2) is void for $A_0(t,u)$ and $B_0(t,u)$ in virtue of (H3) and linearity, and is fulfilled for $F(u)$ and $G(u)$ by their linear growth. Likewise, (H1) is immediate from linearity and Lipschitz continuity.

		\textbf{Step 2}: bootstrapping time-regularity.

		A particular consequence of~\eqref{eq:sl-var-bound} is $F(u) \in \L^p(\Omega \times (0,T); \H^{-1}(\Dom))$ and $G(u) \in \L^p(\Omega \times (0,T); \HS(\ell^2,\L^2(\Dom)))$, with the estimate
		\begin{align}
			\label{eq:bound_semilinearities}
			\| F(u) \|_{\L^p(\Omega \times (0,T); \H^{-1}(\Dom))} + \| G(u) \|_{\L^p(\Omega \times (0,T); \HS(\ell^2,\L^2(\Dom)))} \lesssim 1 + \| u_0 \|_{\L^p(\Omega; \L^2(\Dom))}.
		\end{align}
		Consequently, if we freeze $u$ in the semi-linearities of~\eqref{eq:sl}, that is to say, if we consider~\eqref{eq:sl} as a linear problem with right-hand sides $f = F(u)$ and $g = G(u)$, then Theorem~\ref{thm:main1} becomes applicable (of course we can restrict $p$ from Step~1 further so that the smallness condition in the theorem is verified) and yields a unique solution $v \in \L^p(\Omega \times (0,T); \H^1_0(\Dom)) \cap \L^p(\Omega; \Cont([0,T]; \B^{1-\nicefrac{2}{p}}_{2,p,0}(\Dom)))$.
		For the first estimate in the theorem calculate
		\begin{align}
			&\| v \|_{\L^p(\Omega; \Cont([0,T]; \B^{1-\nicefrac{2}{p}}_{2,p,0}(\Dom)))} \\
			\lesssim{} &\| u_0 \|_{\L^p(\Omega; \B^{1-\nicefrac{2}{p}}_{2,p,0}(\Dom))} + \| F(u) \|_{\L^p(\Omega \times (0,T), \H^{-1}(\Dom))} + \| G(u) \|_{\L^p(\Omega \times (0,T), \HS(\ell^2, \L^2(\Dom)))} \\
			\lesssim{} &1 + \| u_0 \|_{\L^p(\Omega; \B^{1-\nicefrac{2}{p}}_{2,p,0}(\Dom))},
		\end{align}
		where we used~\eqref{eq:bound_semilinearities} and the embedding $\B^{1-\nicefrac{2}{p}}_{2,p,0}(\Dom) \subseteq \L^2(\Dom)$ in the last step.
		The other estimates in the theorem follow also from Theorem~\ref{thm:main1} using the same argument.
		By uniqueness with $p=2$, $u = v$, which completes the proof.
	\end{proof}

	\subsection{Uniform bounds with $p>2$}
	\label{subsec:unif_bounds1}

	Recall the family of solutions $$(u_m)_m \subseteq \L^2(\Omega; \Cont([0,T]; \L^2(\Dom))) \cap \L^2(\Omega \times (0,T); \H^1_0(\Dom))$$ to the approximate problems~\eqref{eq:ql_m} from Lemma~\ref{lem:ql-reg}. Using the result from Section~\ref{subsec:semi}, we derive uniform $\L^p$-bounds for some $p>2$.

	\begin{proposition}
		\label{prop:unif_bounds1}
		Suppose Assumption~\ref{ass:ql_lip} holds. Then there is some $p_0 > 2$ such that for all $m$, for all $p\in (2,p_0]$, and any $u_0 \in \L^p_{\cF_0}(\Omega; \B^{1-\nicefrac{2}{p}}_{2,p,0}(\Dom))$ there is a unique solution $$u_m \in \L^p(\Omega \times (0,T); \H^1_0(\Dom)) \cap \L^p(\Omega; \Cont([0,T]; \B^{1-\nicefrac{2}{p}}_{2,p,0}(\Dom)))$$ to~\eqref{eq:ql_m} satisfying the estimate
		\begin{align}
			\label{eq:qlm_bound1}
			\| u_m \|_{\L^p(\Omega; \Cont([0,T]; \B^{1-\nicefrac{2}{p}}_{2,p,0}(\Dom)))} \lesssim 1 + \| u_0 \|_{\L^p(\Omega; \B^{1-\nicefrac{2}{p}}_{2,p,0}(\Dom))},
		\end{align}
		and for all $\theta \in [0,\nicefrac{1}{2})$ the estimate
		\begin{align}
			\label{eq:qlm_bound2}
			\| u_m \|_{\L^p(\Omega; \H^{\theta,p}(0,T; \H^{1-2\theta}_0(\Dom)))} \lesssim_\theta 1 + \| u_0 \|_{\L^p(\Omega; \B^{1-\nicefrac{2}{p}}_{2,p,0}(\Dom))}.
		\end{align}
		If $\theta \in (\nicefrac{1}{p}, \nicefrac{1}{2})$, then one has moreover the estimate
		\begin{align}
			\label{eq:qlm_bound3}
			\| u_m \|_{\L^p(\Omega; \Cont^{\theta-\nicefrac{1}{p}}([0,T]; \H^{1-2\theta}_0(\Dom)))} \lesssim_\theta 1 + \| u_0 \|_{\L^p(\Omega; \B^{1-\nicefrac{2}{p}}_{2,p,0}(\Dom))}.
		\end{align}
	\end{proposition}

	\begin{proof}
		Fix some $m$. As already mentioned in Section~\ref{subsec:semi}, we can consider the coefficients $\boldsymbol{a}^{\alpha\beta}_{ij}(t,x) \coloneqq {}_m a^{\alpha\beta}_{ij}(t,x,u_m(t,x))$ and $\boldsymbol{b}^{\alpha\beta}_{n,j}(t,x) \coloneqq {}_m b^{\alpha\beta}_{n,j}(t,x,u_m^\beta(t,x))$, which then satisfy Assumption~\ref{ass:secondorder-semi}. Now by Proposition~\ref{prop:sl_higher_int},~\eqref{eq:sl} with these coefficients admits a unique solution $v$ that satisfies the bounds claimed in the proposition. But since
		\begin{align}
			v \in{} &\L^p(\Omega \times (0,T); \H^1_0(\Dom)) \cap \L^p(\Omega; \Cont([0,T]; \B^{1-\nicefrac{2}{p}}_{2,p,0}(\Dom))) \\
			\subseteq{} &\L^2(\Omega \times (0,T); \H^1_0(\Dom)) \cap \L^2(\Omega; \Cont([0,T]; \L^2(\Dom))),
		\end{align}
		it follows $v = u_m$ by uniqueness for $p=2$ stated in Lemma~\ref{lem:ql-reg}, which completes the proof.
	\end{proof}

	\subsection{Stochastic compactness argument}
	\label{subsec:comp_arg}

In Theorem \ref{thm:varcomp} we have already seen that a general tightness result can be deduced from a priori estimates such as the ones of Proposition \ref{prop:unif_bounds1}. In the application to the quasi-linear system we need a slight variation of the tightness result in which a certain weak compactness is also taken into account.

In this section we suppose Assumption~\ref{ass:ql_lip} holds, and we fix $p_0$, $p \in (2, p_0]$, and $(u_m)_{m\geq 1}$ as in Proposition~\ref{prop:sl_higher_int}.
Put
	\begin{align}
		\X_u \coloneqq \Cont([0,T]; \L^2(\Dom)) \cap \L^p_\weak(0,T; \H^1_0(\Dom)), \quad \X_W \coloneqq \Cont([0,T]; \BMSpace_0), \quad \X \coloneqq \X_u \times \X_W.
	\end{align}
	With the subscript \enquote{$\weak$} we indicate that the space $\L^p(0,T; \H^1_0(\Dom))$ is equipped with the weak topology. As a consequence of separability, this has no consequence for questions of measurability.
	The space $\BMSpace_0 \supseteq \BMSpace$ is chosen in such a way that the cylindrical Brownian motion converges almost surely.
	The random vectors $(u_m, W)$ take values in $\X$. Write $\Law_m$ for their (joint) laws.

	We claim that the family $( \Law_m )_m$ is tight. It is sufficient to show tightness of the laws $\Law(u_m)$ on $\X_u$. Let us emphasise that no ``diagonal structure'' of the coefficients is needed.

	\begin{lemma}[Tightness]\label{lem:tightness}
The family of laws $( \Law(u_m) )_m$ on $\X_u$ is tight.
	\end{lemma}

	\begin{proof}
Let $\theta \in (\nicefrac{1}{p}, \nicefrac{1}{2})$. For brevity, put $$\X_u^{\theta,p} \coloneqq \Cont^{\theta-\nicefrac{1}{p}}([0,T]; \H^{1-2\theta}(\Dom)) \cap \L^p(0,T; \H^1_0(\Dom))$$
		and write $\B_R$ for the open ball of radius $R$ in $\X_u^{\theta,p}$.
		First, by Chebyshev's inequality and Proposition~\ref{prop:unif_bounds1},
		\begin{align}
			\label{eq:tight}
			\begin{split}
			\P( \| u_m \|_{\X_u^{\theta,p}} \geq R) &\leq R^{-p} \Exp\| u_m \|_{\X_u^{\theta,p}}^p \\
			&\lesssim R^{-p} \Bigl(1 + \| u_0 \|_{\L^p(\Omega; \B^{1-\nicefrac{2}{p}}_{2,p,0}(\Dom))}^p \Bigr) \to  0 \quad\text{ as } R \to \infty.
			\end{split}
		\end{align}
		Second, we claim that the closure of $\B_R$ is compact in $\X_u$. Indeed, pre-compactness in $\Cont([0,T]; \L^2(\Dom))$ follows from the vector-valued Arz\'ela--Ascoli theorem \cite[Theorem III.3.1]{Lang}, taking into account the compact embedding $\H^1_0(\Dom) \subseteq \L^2(\Dom)$, and compactness in the space $\L^p_\weak(0,T; \H^1_0(\Dom))$ is clear, for $\B_R$ is bounded in the norm topology.
	\end{proof}

	Therefore, Prokhorov's theorem allows to pass to a weakly convergent subsequence (for convenience, we use the same symbol for this subsequence). Second, the Jakubowski--Skorohod theorem~\cite[Thm.~2.7.1]{Hofmanovaetal} gives the following almost sure convergence.

	\begin{lemma}[Skorohod]
		\label{lem:skorohod}
		There exists a probability space $(\tilde \Omega, \tilde \cF, \tilde \P)$, a sequence of $\X$-valued random variables $(\wt{u}_m, \wt{W}_m)$, and a limiting $\X$-valued random variable $(\wt u, \wt W)$, such that
		\begin{enumerate}
			\item for each $m$, the law of $(\wt{u}_m, \wt{W}_m)$ under $\tilde \P$ coincides with the law $\Law_m$, and \label{it:skorohod1}
			\item one has $\tilde \P$-almost surely that $\wt{u}_m \to \wt u$ in $\Cont([0,T]; \L^2(\Dom))$,
			$\nabla \wt{u}_m \to \nabla \wt u$ weakly in $\L^p(0,T; \L^2(\Dom))$, and $\wt{W}_m \to \wt{W}$ in $\Cont([0,T]; \BMSpace_0)$ as $m\to \infty$. \label{it:skorohod2}
		\end{enumerate}
	\end{lemma}

	\begin{corollary}
		\label{cor:skorohod}
		The inclusion $$( \wt{u}_m )_m \subseteq \L^p(\wt \Omega; \Cont([0,T]; \L^2(\Dom))) \cap \L^p(\wt \Omega \times (0,T); \H^1_0(\Dom))$$ holds with the uniform bound
		\begin{align}
			\label{eq:wt_um_unif}
			\sup_m \left( \| \wt{u}_m \|_{\L^p(\wt \Omega; \Cont([0,T]; \L^2(\Dom)))} + \| \wt{u}_m \|_{\L^p(\wt \Omega \times (0,T); \H^1_0(\Dom))} \right) \lesssim 1 + \| u_0 \|_{\L^p(\Omega; \B^{1-\nicefrac{2}{p}}_{2,p,0}(\Dom))}.
		\end{align}
	\end{corollary}

	According to~\cite[Lem.~4.8]{DHV}, $\wt{W}$ is a cylindrical Brownian motion.
	Fix $\alpha = 1, \dots, N$ and $\xi\in \Cont^\infty_c(\Dom)$. Define
	\begin{align}
		\wt{M}^\alpha(t) &= \lb  \wt{u}^\alpha(t), \xi\rb - \lb  \wt{u}^\alpha(0), \xi\rb + \int_{0}^t \lb  a^{\alpha\beta}_{ij}(\wt{u}) \partial_j \wt{u}^\beta, \partial_i \xi\rb \d r \\ & \qquad + \int_{0}^t \lb \Phi^\alpha_i(\wt{u}), \partial_i \xi\rb \d r - \int_{0}^t \lb \phi^\alpha(\wt{u}), \xi\rb \d r.
	\end{align}
	As presented in the proof of~\cite[Prop.~4.7]{DHV}, well-posedness of~\eqref{eq:ql} under Assumption~\ref{ass:ql_lip}, that is to say, validity of Theorem~\ref{thm:quasi1}, follows directly from the following lemma.

	\begin{lemma}
		\label{lem:martingales1}
		The processes
		\begin{align}\label{eq:martingaletobe}
			\begin{aligned}
				\wt{M}, & \  \wt{M}^2  - \sum_{n\geq 1} \int_0^\cdot  \big[\lb b^{\alpha\beta}_{n,j}(\wt{u}^\beta) \partial_j \wt{u}^\beta,\xi\rb + \lb g^\alpha_{n}(\wt{u}), \xi\rb\big]^2 \d r, \\ & \ \text{and} \ \wt{M} w_{n} - \int_0^\cdot \lb b^{\alpha\beta}_{n,j}(\wt{u}^\beta) \partial_j \wt{u}^\beta,\xi\rb + \lb g^\alpha_{n}(\wt{u}), \xi\rb \d r
			\end{aligned}
		\end{align}
		are $(\wt{\cF}_t)$-martingales, where $\wt{\cF}_t \coloneqq \sigma(\{\wt{u}(s), \wt{w}_n(s) \colon s\leq t, n\geq 1\})$.
	\end{lemma}

	\begin{proof}
		In a first step, we showcase the general strategy for the process $\wt{M}$. Afterwards, we explain the necessary changes for the two remaining processes.

		\textbf{Step 1}: $\wt{M}$ is an $(\wt{\cF}_t)$-martingale.

		Introduce the processes
		\begin{align}
			\label{eq:martingale_candidates}
			M^\alpha_m(t) &= \lb  u^\alpha_m(t), \xi\rb - \lb  u^\alpha_0, \xi\rb + \int_{0}^t \lb  {}_m a^{\alpha\beta}_{ij}(R_m u_m) \partial_j u^\beta_m, \partial_i \xi\rb \d r \\ & \qquad + \int_{0}^t \lb \Phi^\alpha_i(u_m), \partial_i \xi\rb \d r - \int_{0}^t \lb \phi^\alpha(u_m), \xi\rb \d r,
			\\
			\wt{M}^\alpha_m(t) &= \lb  \wt{u}^\alpha_m(t), \xi\rb - \lb  \wt{u}^\alpha_m(0), \xi\rb + \int_{0}^t \lb  {}_m a^{\alpha\beta}_{ij}(R_m \wt{u}_m) \partial_j \wt{u}^\beta_m, \partial_i \xi\rb \d r \\ & \qquad + \int_{0}^t \lb \Phi^\alpha_i(\wt{u}_m), \partial_i \xi\rb \d r - \int_{0}^t \lb \phi^\alpha(\wt{u}_m), \xi\rb \d r.
		\end{align}
		For each $m$, $M^\alpha_m$ is an $(\cF_t)$-martingale because $u_m$ is a solution to~\eqref{eq:ql_m}.

		Fix $\alpha$ and $s\leq t$ throughout the proof. Let $\rho_s$ denote the canonical restriction $\Cont([0,T]; \X) \to \Cont([0,s]; \X)$ and fix any continuous function $\gamma \colon \Cont([0,s]; \X) \to [0,1]$. One has that $$\gamma_m \coloneqq \gamma(\rho_s(u_m, W))$$ is $\cF_s$-measurable. Therefore, conditioning and the martingale property yield
		\begin{align}
			\label{eq:mart1}
			\Exp \Bigl(\gamma_m [M^\alpha_m(t) - M^\alpha_m(s)]\Bigr) = 0.
		\end{align}
		We claim that $\gamma_m$, $M^\alpha_m(t)$, and $M^\alpha_m(s)$ depend on $u$ in a measurable way. Indeed, these quantities even depend continuously on $u$ if we equip $\L^p(0,T; \H^1_0(\Dom))$ with the strong topology in the definition of $\X$. As already mentioned, passing to the weak topology preserves measurability, which gives the claim. Therefore, with Lemma~\ref{lem:skorohod}~\ref{it:skorohod1} deduce from~\eqref{eq:mart1} that
		\begin{align}
			\label{eq:mart2}
			\wt\Exp \Bigl(\wt{\gamma}_m [\wt{M}^\alpha_m(t) - \wt{M}^\alpha_m(s)]\Bigr) = 0,
		\end{align}
		where $\wt{\gamma}_m$ and $\wt{M}^\alpha_m$ are defined by the same expressions as $\gamma_m$ and $M^\alpha_m$, but with $u_m$ replaced by $\wt{u}_m$ and $W$ replaced by $\wt{W}$.

		Next, we take the limit $m\to \infty$ in~\eqref{eq:mart2}. To do so, we appeal to Vitali's convergence theorem. Owing to Corollary~\ref{cor:skorohod}, we have for instance
		\begin{align}
			\label{eq:mart_vitali1}
			\wt\Exp \left|\int_{0}^t \lb  {}_m a^{\alpha\beta}_{ij}(R_m \wt{u}_m) \partial_j \wt{u}^\beta_m, \partial_i \xi\rb \d r\right|^p &\lesssim T^{p-1} \wt\Exp \int_0^T \| \nabla \wt{u}_m \|_{\L^2(\Dom)}^p \\
			&\lesssim 1 + \| u_0 \|_{\L^p(\Omega; \B^{1-\nicefrac{2}{p}}_{2,p,0}(\Dom))}^p,
		\end{align}
		which justifies the application of Vitali's convergence theorem.
		It remains to determine the almost-sure limit of $\wt{\gamma}_m [\wt{M}^\alpha_m(t) - \wt{M}^\alpha(t)]$. On the one hand, by continuity of $\gamma \circ \rho_s$ on $\X$, deduce $\wt{\gamma}_m \to \wt{\gamma}$ almost surely from Lemma~\ref{lem:skorohod}~\ref{it:skorohod2}. On the other hand, we claim that $\wt{M}^\alpha_m(t) \to \wt{M}^\alpha(t)$ almost surely. Convergence of the first two terms of $\wt{M}^\alpha_m(t)$ is immediate. For the fourth and fifth term, use that $\Phi$ and $\phi$ are Lipschitz continuous along with almost sure convergence in $\Cont([0,T]; \L^2(\Dom)) \subseteq \L^1(0,T; \L^2(\Dom))$. The most challenging term is the third. Rewrite
		\begin{align*}
			&\int_{0}^t \lb {}_m a^{\alpha\beta}_{ij}(R_m \wt{u}_m) \partial_j \wt{u}^\beta_m, \partial_i \xi\rb \d r - \int_{0}^t \lb a^{\alpha\beta}_{ij}(\wt{u}) \partial_j \wt{u}^\beta, \partial_i \xi\rb \d r \\
			={} &\int_{0}^t \lb \partial_j \wt{u}^\beta_m, ({}_m a^{\alpha\beta}_{ij}(R_m \wt{u}_m) -  a^{\alpha\beta}_{ij}(\wt{u})) \partial_i \xi\rb \d r + \int_{0}^t \lb \partial_j (\wt{u}^\beta_m - \wt{u}^\beta), a^{\alpha\beta}_{ij}(\wt{u}) \partial_i \xi\rb \d r \\
			={} &\I + \II.
		\end{align*}
		We use the convergence properties from Lemma~\ref{lem:skorohod}~\ref{it:skorohod2} as follows to conclude: using the weak convergence of $\nabla \wt{u}_m$ in $\L^2(0,t;\L^2(\Dom))$ we see that $\II \to 0$ as $m\to \infty$ almost surely. Moreover, $\nabla \wt{u}_m$ is uniformly bounded in $\L^p(0,T; \L^2(\Dom)) \subseteq \L^2(0,t; \L^2(\Dom))$ almost surely. Consequently, $\I$ goes to zero if $({}_m a^{\alpha\beta}_{ij}(R_m \wt{u}_m) -  a^{\alpha\beta}_{ij}(\wt{u})) \partial_i \xi \to 0$ in $\L^2(0,t; \L^2(\Dom))$. Split further
		\begin{align}
			{}_m a^{\alpha\beta}_{ij}(R_m \wt{u}_m) - a^{\alpha\beta}_{ij}(\wt{u})
			={} ({}_m a^{\alpha\beta}_{ij}(R_m \wt{u}_m) - a^{\alpha\beta}_{ij}(R_m \wt{u}_m)) + (a^{\alpha\beta}_{ij}(R_m \wt{u}_m) - a^{\alpha\beta}_{ij}(\wt{u})).
		\end{align}
		Since the coefficients are uniformly bounded and $\xi$ is smooth and compactly supported, we have Vitali's convergence theorem in $(r,x)$ at our disposal.
Up to passing to another subsequence, we can suppose $R_m \wt{u}_m \to \wt{u}$ pointwise almost everywhere in $(r,x)$.
In particular, for almost every $(r,x)$, the closure of the set $\{ R_m \wt{u}_m(r,x) \colon m \in \N \}$ is compact.
Then, the claim follows by continuity of $a^{\alpha\beta}_{ij}$ in conjunction with the convergence ${}_m a^{\alpha\beta}_{ij} \to a^{\alpha\beta}_{ij}$ uniformly on compact subsets.

		In conclusion, we deduce
		\begin{align}
			\label{eq:mart3}
			\wt\Exp\Bigl(\wt{\gamma} [\wt{M}^\alpha(t) - \wt{M}^\alpha(s)]\Bigr) = 0.
		\end{align}
		But~\eqref{eq:mart3} remains valid if we replace $\wt{\gamma}$ by \emph{any} bounded and $\wt{\cF}_s$-measurable function in virtue of the monotone class theorem, since $\wt{\cF}_s$ is by definition the $\sigma$-field induced by $\wt{u}$ and $\wt{W}$. It follows that $\wt{M}^\alpha$ is an $(\wt{\cF}_t)$-martingale as claimed.

		\textbf{Step 2}: the remaining two processes.

		The general strategy is as in Step~1 and we only discuss the necessary modifications. The martingale property for the processes related to $u_m$ is again clear. We follow the proof of Step~1 until we reach the analogues of~\eqref{eq:mart2}. The first step in the identification of the limit, the moment condition~\eqref{eq:mart_vitali1}, can be reused with the following observations: first of all, since the moment condition for one term holds with $p>2$, any product appearing in $(\wt{M}_m^\alpha)^2$ satisfies a moment condition with $\nicefrac{p}{2} > 1$ in virtue of H\"older's inequality. The same is true for $\wt{M}_m w_k$, of course. Likewise, one has
		\begin{align}
			\label{eq:mart_vitali2}
			\wt\Exp \left|\int_0^t  \lb {}_m b^{\alpha\beta}_{n,j}(R_m \wt{u}_m^\beta) \partial_j \wt{u}_m^\beta,\xi\rb^2 \d r \right|^{\nicefrac{p}{2}} &\lesssim T^{\nicefrac{p}{2}-1} \wt\Exp \int_0^T \| \nabla \wt{u}_m \|_{\L^2(\Dom)}^p \\
			&\lesssim 1 + \| u_0 \|_{\L^p(\Omega; \B^{1-\nicefrac{2}{p}}_{2,p,0}(\Dom))}^p.
		\end{align}
		It remains to verify almost sure convergence, for instance
		\begin{align}
			\label{eq:conv_b}
			\int_0^t  \lb {}_m b^{\alpha\beta}_{n,j}(R_m \wt{u}_m^\beta) \partial_j \wt{u}_m^\beta,\xi\rb^2 \d r \to \int_0^t  \lb b^{\alpha\beta}_{n,j}(\wt{u}^\beta) \partial_j \wt{u}^\beta,\xi\rb^2 \d r,
		\end{align}
		for fixed $n$.
		First, we claim that we can replace ${}_m b^{\alpha\beta}_{n,j}(R_m \wt{u}_m^\beta)$ by $b^{\alpha\beta}_{n,j}(\wt{u}^\beta)$ on the left-hand side of~\eqref{eq:conv_b}. Indeed, the calculation is similar to the treatment of term~$\I$ in Step~1, but one has to use uniform boundedness of $\nabla \wt{u}_m$ in $\L^p(0,T; \L^2(\Dom))$ with $p>2$.
		Next, define the auxiliary function
		\begin{align}
			\label{eq:aux_antideriv}
			\bar{b}_{n,j}^{\alpha\beta}(r,x,y) \coloneqq \int_0^y b_{n,j}^{\alpha\beta}(r,x,z) \d z, \qquad (y\in \R).
		\end{align}
		By the chain rule, $\partial_j \bar{b}_{n,j}^{\alpha\beta}(r,x,\wt{u}(r,x)) = b_{n,j}^{\alpha\beta}(r,x,\wt{u}(r,x)) \partial_j \wt{u}(r,x)  + \Theta^{\alpha\beta}_{n,j}(r,x,\wt{u}^{\beta}(r,x))$,
where $\Theta^{\alpha\beta}_{n,j}(r,x,y) = \int_0^{y} \partial_j b_{n,j}^{\alpha\beta}(r,x,z) \d z$. Therefore, to prove \eqref{eq:conv_b} it is enough to show
		\begin{align}
			\int_0^t  \lb \bar{b}^{\alpha\beta}_{n,j}(\wt{u}^\beta) - \bar{b}^{\alpha\beta}_{n,j}(\wt{u}_m^\beta), \partial_j \xi\rb^2 \d r \to 0 \ \ \text{and} \ \ \int_0^t \lb \Theta^{\alpha\beta}_{n,j}(\cdot,\wt{u}^{\beta}) - \Theta^{\alpha\beta}_{n,j}(\cdot,\wt{u}_m^{\beta}), \xi\rb^2 \d r\to 0.
		\end{align}
		For the first term, using Vitali's convergence theorem once more, we only need to show convergence for a fixed time $r\in (0,t)$. Now since $\wt{u}_m^\beta\to \wt{u}^\beta$ in $\L^2(\Dom)$, we can conclude the proof by passing to another subsequence and using continuity of $\bar{b}^{\alpha\beta}_{n,j}$. For the second term, the uniform boundedness of $\partial_j b_{n,j}^{\alpha\beta}$ implies that $|\Theta^{\alpha\beta}_{n,j}(\cdot,\wt{u}^{\beta}) - \Theta^{\alpha\beta}_{n,j}(\cdot,\wt{u}_m^{\beta})|\lesssim |\wt{u}^{\beta} - \wt{u}^{\beta}_m|$. Therefore, the desired convergence follows from $\wt{u}_m\to \wt{u}$ in $\L^2(0,t;\L^2(\Dom))$. 

		Let us stress that the last argument heavily relies on the structural assumption on $b$, and is the reason why we cannot allow full non-linear dependence on all components of $u$ as is the case for $a$.
	\end{proof}

	\begin{remark}
		We stress that the above proof critically uses the $p$-integrability of $\nabla \wt{u}_m$ in $\wt\Omega \times (0,T)$. Indeed, to get $\L^1(\Omega)$-convergence of \eqref{eq:mart_vitali2} we used Vitali's convergence theorem based on the uniform boundedness of $\nabla \wt{u}_m$ in $\L^p(\Omega \times (0,T); \L^2(\Dom))$ with $p>2$.
That such a bound is available for $p>2$ is a novelty of our approach.
	\end{remark}

	\section{Quasilinear problem with growing non-linearities}
	\label{sec:ql_growing}

	We reduce the general case with growing non-linearities to the Lipschitz case from the last section. This happens using a truncation argument. We use the following classical result from convex analysis (see \cite[Proposition 5.3]{B11}) to include systems in a neat way.

	\begin{lemma}[Phelps]
		\label{lem:projection}
		Let $C$ be a non-empty, closed, convex subset of a Hilbert space $K$. Then the projection onto $C$ is Lipschitz continuous with Lipschitz constant $1$.
	\end{lemma}

	With the preceding lemma, we can consider truncated non-linearities $\Phi \circ P$ and $\phi \circ P$, where $P$ is a projection to a bounded and convex set. As in Section~\ref{subsec:semi}, we would like to consider $f=\Phi_i (P(u)) + \phi(P(u))$ as a fixed right-hand side. However, that $f$ is an admissible right-hand side is more difficult with growing non-linearities. We will take care of this in Proposition~\ref{prop:ito_trace}. The next two lemmas are a preparation for the last-mentioned proposition.

	\begin{lemma}[{\cite[Lem.~8]{DMS}}]
		\label{lem:p-approx}
		For $q>2$ and $m\geq 1$, the function $\psi_m \colon \R \to \R$ given by
		\begin{align}
			\psi_m(\xi) = \begin{cases}
				|\xi|^q, & |\xi| \leq m, \\
				m^{q-2} \Bigl[ \frac{q(q-1)}{2} \xi^2 -q(q-2)m |\xi| + \frac{(q-1)(q-2)}{2} m^2\Bigr], & |\xi| > m
			\end{cases}
		\end{align}
		is twice continuously differentiable, has bounded second derivative, and satisfies the following properties:
		\begin{align}
			|\psi_m'(\xi)| &\leq |\xi|\psi_m''(\xi), \label{eq:phim_lingro} \\
			\frac{\psi_m'(\xi)}{\xi} &\geq 0, \mathrlap{\quad (\xi \neq 0),} \label{eq:phim_derpos} \\
			\psi_m''(\xi) &\leq q(q-1) (1+\psi_m(\xi)), \label{eq:phim_snd_b1}  \\
			\xi^2 \psi_m''(\xi) &\leq q(q-1) \psi_m(\xi), \label{eq:phim_snd_b2} \\
			\psi_m''(\xi_1) &\leq \psi_m''(\xi_2), \mathrlap{\quad (|\xi_1| \leq |\xi_2|).} \label{eq:phim_snd_inc}
		\end{align}
		Moreover, $\psi_m \to |\cdot|^q$ and $\psi_m'' \to |\cdot|^{q-2}$ pointwise.
	\end{lemma}

	The following technical lemma relates the dissipativity of $\phi$ stated in Assumption~\ref{ass:ql}~\ref{ass:2ndorder14} with the auxiliary function $\psi_m$.

	\begin{lemma}
		\label{lem:dissip}
		Let $v\in \R^N$. For all $q \in (2, \infty)$ and $m \geq 1$ one has for $\psi_m$ defined as in Lemma~\ref{lem:p-approx} the estimate
		\begin{align}
			\sum_\alpha \psi_m'(v^\alpha) \phi^\alpha(v) \lesssim_q 1 + \sum_\beta \psi_m(v^\beta).
		\end{align}
	\end{lemma}

	\begin{proof}
		Let $v\in \R^N$ and fix $\alpha$ for the moment. We can assume $v^\alpha \neq 0$. Calculate using~\eqref{eq:phim_derpos}, Assumption~\ref{ass:ql}~\ref{ass:2ndorder14}, and~\eqref{eq:phim_lingro} that
		\begin{align}
			\psi_m'(v^\alpha) \phi^\alpha(v) = \frac{\psi_m'(v^\alpha)}{v^\alpha} v^\alpha \phi^\alpha(v) \lesssim \frac{\psi_m'(v^\alpha)}{v^\alpha} (1 + |v|^2) \lesssim \psi_m''(v^\alpha) (1 + |v|^2).
		\end{align}
		Using~\eqref{eq:phim_snd_b1}, the term $\psi_m''(v^\alpha)$ can be controlled by $1 + \psi_m(v^\alpha)$. Likewise, $\psi_m''(v^\alpha) |v^\alpha|^2$ is controlled by $\psi_m(v^\alpha)$ owing to~\eqref{eq:phim_snd_b2}. It remains to consider $\psi_m''(v^\alpha) |v^\beta|^2$. We distinguish the cases $|v^\alpha| \leq |v^\beta|$ and $|v^\alpha| \geq |v^\beta|$. In the first case, we use that $\psi_m''$ is increasing,~\eqref{eq:phim_snd_inc}, to reduce to the known case $\psi_m''(v^\beta) |v^\beta|^2$. Likewise, we reduce in the second case using that $|\cdot|^2$ is increasing. Finally, summing in $\alpha$ gives the claim.
	\end{proof}

	\begin{proposition}[Bootstrapping integrability]
		\label{prop:ito_trace}
		Suppose Assumption \ref{ass:ql} holds, and for the moment suppose that there is a constant $L$ such that
\[|\phi(t,x,y)|+ |\Phi(t,x,y)|\leq L(1+|y|), \ \ \ t\in [0,T],x\in D,y\in \R^N.\]
 Let $u_0\in \L^{q}_{\cF_0}(\Omega \times \Dom)$ with $q\in [2, \infty)$ fixed and let $v\in \L^2(\Omega\times (0,T); \H^1_0(\Dom)) \cap \L^2(\Omega; \Cont([0,T]; \L^2(\Dom)))$ be a solution to \eqref{eq:ql} with initial datum $u_0$. Then one has
		\begin{align}
			\|v\|_{\L^q(\Omega \times (0,T) \times \Dom)}^q + \Exp \int_0^T \int_{\Dom} |v|^{q-2}  |\nabla v|^2 \d x \d s \lesssim 1+\|u_0\|_{\L^q(\Omega;\L^q(\Dom))}^q,
		\end{align}
where the implicit constant does not depend on $L$.
	\end{proposition}
The additional growth condition on $\phi$ and $\Phi$ are needed to make sure that the functions in the proof below are integrable. Later on, we will apply the lemma to truncated versions of $\phi$ and $\Phi$, and therefore the growth condition does not lead to additional assumptions.

	\begin{proof}
		Fix $\alpha$ and let $m \geq 1$, $t\in [0,T]$. Define the linear functional $v \mapsto \int_\Dom \psi_m(v) \d x$ on $\L^2(\Dom)$. Since $\psi_m''$ is bounded and, taking~\eqref{eq:phim_lingro} into account, $\psi_m'$ is of linear growth, the It\^o formula from~\cite[Thm.~4.2]{Pardoux} is applicable and yields
		\begin{align}
			\int_\Dom \psi_m(v^\alpha(t)) \d x &= \int_\Dom \psi_m(v^\alpha(0)) \d x \\
			&\quad -\int_0^t \int_\Dom \psi_m''(v^\alpha) a_{ij}^{\alpha}(v) \partial_i v^\alpha \partial_j v^\alpha \d x \d s \\
			&\quad -\int_0^t \int_\Dom \psi_m''(v^\alpha) \Phi_i^\alpha(v) \partial_i v^\alpha \d x \d s \\
			&\quad +\int_0^t \int_\Dom \psi_m'(v^\alpha) \phi^\alpha(v) \d x \d s \\
			&\quad + \sum_n \int_0^t \int_\Dom \psi_m'(v^\alpha) \bigl[ b_{n,j}^{\alpha}(v^\alpha) \partial_j v^\alpha + g_n^\alpha(v) \bigr] \d x \d w_n \\
			&\quad + \sum_n \frac{1}{2} \int_0^t \int_\Dom \psi_m''(v^\alpha) \Bigl| \sum_j b_{n,j}^{\alpha}(v^\alpha) \partial_j v^\alpha + g_n^\alpha(v) \Bigr|^2 \d x \d s \\
			&\quad = \I - \II - \III + \IV + \VV + \VI.
		\end{align}
Recall that $\Phi^{\alpha}(s,x,y) = \wh{\Phi}^\alpha(s,x,y) + \overline{\Phi}^{\alpha}(y^{\alpha})$. Therefore, the part of $\III$ involving $\overline{\Phi}^{\alpha}$ vanishes as can be seen in the same way as in \cite[Lemma 3.5]{AVdissi}. Indeed, 
\begin{align*}
\int_\Dom \psi_m''(v^\alpha) \overline{\Phi}_i^\alpha(v^{\alpha}) \partial_i v^\alpha \d x = \int_{\Dom} \div_x \int_0^{v^\alpha} \psi_m''(r) \overline{\Phi}_i^\alpha(r)  \d r \d x =0,
\end{align*}
where we applied the divergence theorem and the fact that $v^{\alpha}$ vanishes at the boundary. Thus $\III$ can be bounded as follows
\begin{align*}
-\III&\leq \Big|\int_0^t \int_\Dom \psi_m''(v^\alpha) \wh{\Phi}_i^\alpha(v) \partial_i v^\alpha \d x \d s\Big|\\ & \leq C_{\delta} \int_0^t \int_\Dom \psi_m''(v^\alpha) |\wh{\Phi}^\alpha(v)|^2 \d x \d s + \delta \int_0^t \int_\Dom \psi_m''(v^\alpha) |\nabla v^\alpha|^2 \d x \d s \\ & =: \III(\wh{\Phi}^\alpha(v)) + \III(\nabla v^{\alpha}).
\end{align*}

		Use the inequality $\nicefrac{1}{2} |b+g|^2 \leq (\nicefrac{1}{2} + \eps) b^2 + C_\eps g^2$ for all $\eps > 0$ in~$\VI$ to get
		\begin{align}
			\VI &\leq (\nicefrac{1}{2} + \eps) \sum_n \int_0^t \int_\Dom \psi_m''(v^\alpha) \Bigl| \sum_j b_{n,j}^{\alpha}(v^\alpha) \partial_j v^\alpha \Bigr|^2 \d x \d s \\
			&\qquad+ C_\eps \sum_n \int_0^t \int_\Dom \psi_m''(v^\alpha) |g_n^\alpha(v)|^2 \d x \d s \\
			&= \VI(b) + \VI(g).
		\end{align}
		By the \enquote{diagonal structure} of $a$ and $b$ stated in Assumption~\ref{ass:ql}~\ref{ass:2ndorder10}, the coercivity condition in Assumption~\ref{ass:ql}~\ref{ass:2ndorder12} can be applied componentwise. Therefore, by choosing $\eps$ and $\delta$ small enough (similar to the proof of Lemma~\ref{lem:ql-reg}), we get the lower bound
		\begin{align}
			\II - \III(\nabla v^{\alpha}) - \VI(b) \geq \frac{\lambda}{2} \int_0^t \int_\Dom \psi_m''(v^\alpha) |\nabla v^\alpha|^2 \d x \d s.
		\end{align}
		This is the only argument that uses Assumption~\ref{ass:ql}~\ref{ass:2ndorder10}, see Remark~\ref{rem:diagonal} for further discussion.
		Hence, we can absorb $-\II +\III(\nabla v^{\alpha}) + \VI(b)$ into the left-hand side, to give
		\begin{align}
			\label{eq:after_absorb}
			\begin{split}
				&\int_\Dom \psi_m(v^\alpha) \d x + \frac{\lambda}{2} \int_0^t \int_\Dom \psi_m''(v^\alpha) |\nabla v^\alpha|^2 \d x \d s \\
				\leq{} &\int_\Dom \psi_m(v^\alpha(0)) \d x \\
				&+ \int_0^t \int_\Dom \psi_m'(v^\alpha) \phi^\alpha(v) \d x \d s \\
				&+ \sum_n \int_0^t \int_\Dom \psi_m'(v^\alpha) \bigl[ b_{n,j}^{\alpha}(v^\alpha) \partial_j v^\alpha + g_n^\alpha(v) \bigr] \d x \d w_n \\
				&+ C_\eps \sum_n \int_0^t \int_\Dom \psi_m''(v^\alpha) |g_n^\alpha(v)|^2 \d x \d s+ C_\delta \int_0^t \int_\Dom \psi_m''(v^\alpha) |\wh{\Phi}^{\alpha}(v)|^2 \d x \d s.
			\end{split}
		\end{align}
		The first term on the right-hand side of~\eqref{eq:after_absorb} is controlled by $\| u_0 \|_{\L^q(\Dom)}^q$ by definition of $\psi_m$. We come to the second term. Applying Lemma~\ref{lem:dissip} with $v = v(s,x)$, estimate
		\begin{align}
			\int_0^t \int_\Dom \psi_m'(v^\alpha) \phi^\alpha(v) \d x \d s \lesssim 1 + \int_0^t \int_\Dom \sum_\beta \psi_m(v^\beta) \d x \d s.
		\end{align}
		Since $g$ is of linear growth, we get $\sum_{n}\psi_m''(v^\alpha) |g_n^\alpha(v)|^2 \lesssim \psi_m''(v^\alpha)(1+|v|^2)$, so the proof of Lemma~\ref{lem:dissip} reveals also
		\begin{align}
			\int_0^t \int_\Dom \psi_m''(v^\alpha) |g_n^\alpha(v)|^2 \d x \d s \lesssim 1 + \int_0^t \int_\Dom \sum_\beta \psi_m(v^\beta) \d x \d s.
		\end{align}
		Since $\wh{\Phi}^{\alpha}$ is Lipschitz, a similar estimate holds for $\psi_m''(v^\alpha) |\wh{\Phi}^{\alpha}(v)|^2$.
		So far, we have in summary
		\begin{align}
			\label{eq:before_exp}
			\begin{split}
			&\int_\Dom \psi_m(v^\alpha) \d x + \frac{\lambda}{2} \int_0^t \int_\Dom \psi_m''(v^\alpha) |\nabla v^\alpha|^2 \d x \d s \\
			\lesssim{} &1 + \| u_0 \|_{\L^q(\Dom)}^q + \int_0^t \int_\Dom \sum_\beta \psi_m(v^\beta) \d x \d s \\
			&\quad+ \sum_n \int_0^t \int_\Dom \psi_m'(v^\alpha) \bigl[ b_{n,j}^{\alpha}(v^\alpha) \partial_j v^\alpha + g_n^\alpha(v) \bigr] \d w_n \d s.
			\end{split}
		\end{align}
		We take the expectation, so that the stochastic integral vanishes, and sum in $\alpha$ afterwards, to find
		\begin{align}
			\label{eq:after_exp}
			\begin{split}
			&\Exp\int_\Dom \sum_\alpha \psi_m(v^\alpha) \d x + \frac{\lambda}{2} \Exp \int_0^t \int_\Dom \sum_\alpha \psi_m''(v^\alpha) |\nabla v^\alpha|^2 \d x \d s \\
			\lesssim{} &1 + \Exp \| u_0 \|_{\L^q(\Dom)}^q + \Exp \int_0^t \int_\Dom \sum_\alpha \psi_m(v^\alpha) \d x \d s.
			\end{split}
		\end{align}
		Note that for all $t \in [0,T]$ one has $\Exp \int_0^t \int_\Dom \sum_\alpha \psi_m(v^\alpha(s)) \d x \d s$ is finite, for $\psi_m(v(s)) \leq |v(s)|^2 + C$. Therefore, Gronwall's inequality is applicable with $t \mapsto \Exp \int_\Dom \sum_\alpha \psi_m(v^\alpha(t)) \d x$ and yields
		\begin{align}
			\label{eq:gronwall}
			\Exp\int_\Dom \sum_\alpha \psi_m(v^\alpha) \d x \lesssim 1 + \Exp \| u_0 \|_{\L^q(\Dom)}^q
		\end{align}
		for all $t\in (0,T)$. Integrate this bound in $t$, use Fubini's theorem, and take (in virtue of the pointwise convergence of $\psi_m$ from Lemma~\ref{lem:p-approx}) the limit $m\to \infty$ to conclude the first claim of the proposition.
		Finally, plug~\eqref{eq:gronwall} back into~\eqref{eq:after_exp} to deduce
		\begin{align}
			\frac{\lambda}{2} \Exp \int_0^t \int_\Dom \sum_\alpha \psi_m''(v^\alpha) |\nabla v|^2 \d x \d s
			\lesssim{} &1 + \Exp \| u_0 \|_{\L^q(\Dom)}^q.
		\end{align}
		This time using the pointwise convergence of $\psi_m''$ stated in Lemma~\ref{lem:p-approx}, take once more the limit $m\to \infty$. Afterwards, take the limit $t\to T$ to conclude the second claim in the proposition.
	\end{proof}

	\begin{remark}
		Even more is true in the last proposition: if we take first the supremum over $t$ in~\eqref{eq:before_exp} and then expectations, a calculation based on the Burkholder--Davis--Gundy inequality even yields the estimate
		\begin{align}
			\|v\|_{\L^q(\Omega;\Cont([0,T];\L^q(\Dom)))} \lesssim 1+\|u_0\|_{\L^q(\Omega;\L^q(\Dom))}.
		\end{align}
		We will not need this extra information and decided therefore to stick to the shorter proof that leads to Proposition~\ref{prop:ito_trace}.
	\end{remark}

	\begin{proof}[Proof of Theorem~\ref{thm:quasi2}]
		Write $P_R$ for the projection onto the closed ball of radius $R$ around $0$ in $\R^N$.
		Then, for $R>0$, define the non-linearities ${}_R\phi^\alpha(t,x,y) \coloneqq \phi^\alpha(t,x,P_R(y))$ and ${}_R\Phi^\alpha_{i}(t,x,y) \coloneqq \Phi^\alpha_i(t,x,P_R(y))$. They are Lipschitz continuous and bounded (in particular, of linear growth) with constant depending on $R$, and satisfy Assumptions~\ref{ass:ql}~\ref{ass:2ndorder13} and~\ref{ass:2ndorder14} uniformly in $R$ by virtue of Lemma~\ref{lem:projection}.
		Consider the equation
		\begin{align}
			\label{eq:ql_R}
			\tag{R-QL}
			\left\{\quad
			\begin{aligned}
				\d u^\alpha  &= \big[ \partial_i (a^{\alpha\beta}_{ij}(u) \partial_j u^\beta) + \partial_i ({}_R\Phi^\alpha_{i}(u)) +{}_R\phi^\alpha(u) \big] \d t
				\\  & \qquad \qquad
				+ \sum_{n\geq 1}  \big[ b^{\alpha\beta}_{n,j}(u^\beta) \partial_j u^\beta + g^\alpha_{n}(u) \big] \d w_n,\ \ &\text{ on }\Dom,\\
				u^\alpha &= 0, \ \ &\text{ on }\partial\Dom,
				\\
				u^\alpha(0)&=u^\alpha_{0},\ \ &\text{ on }\Dom.
			\end{aligned}\right.
		\end{align}

		Problem~\eqref{eq:ql_R} fulfills Assumption~\ref{ass:ql_lip}. Therefore, by the well-posedness result of Section~\ref{sec:ql_lip}, there are $p_0 > 2$ and solutions
		$$u_R \in \L^p(\Omega \times (0,T); \H^1_0(\Dom)) \cap \L^p(\Omega; \Cont([0,T]; \B^{1-\nicefrac{2}{p}}_{2,p,0}(\Dom)))$$ to~\eqref{eq:ql_R}.

		Apply Proposition~\ref{prop:ito_trace} to $u_R$ to find $\| u_R \|_{\L^q(\Omega \times (0,T) \times D)} \lesssim 1 + \| u_0 \|_{\L^q(\Omega \times D)}$. With the growth condition for the non-linearities we find, keeping $ph \leq q$ in mind,
		\begin{align}
			\label{eq:moment_cutoff_nonlin}
			&\|{}_R\phi^\alpha(u_R)\|_{\L^p(\Omega\times(0,T);L^2(D))} + \|{}_R\Phi^\alpha_i(u_R)\|_{\L^p(\Omega\times(0,T);L^2(D))} \\
			\lesssim{} &1+\|u_{R}\|^h_{L^{q}(\Omega\times (0,T) \times D)}
			\\ \lesssim{} &1+\|u_0\|^h_{L^{q}(\Omega \times D)}.
		\end{align}
		Therefore, we can freeze $u_R$, ${}_R\Phi$ and ${}_R\phi$ in~\eqref{eq:ql_R} to obtain, as in Section~\ref{subsec:unif_bounds1}, a linear problem to which Theorem~\ref{thm:main1} applies. This results in the uniform bounds
		\begin{align}
			\label{eq:ql_grow_bound1}
			\| u_R \|_{\L^p(\Omega; \Cont([0,T]; \B^{1-\nicefrac{2}{p}}_{2,p,0}(\Dom)))} \lesssim 1 + \| u_0 \|_{\L^p(\Omega; \B^{1-\nicefrac{2}{p}}_{2,p,0}(\Dom))} + \|u_0\|^h_{L^{q}(\Omega \times D)},
		\end{align}
		and for all $\theta \in [0,\nicefrac{1}{2})$ the estimate
		\begin{align}
			\label{eq:ql_grow_bound2}
			\| u_R \|_{\L^p(\Omega; \H^{\theta,p}(0,T; \H^{1-2\theta}_0(\Dom)))} \lesssim_\theta 1 + \| u_0 \|_{\L^p(\Omega; \B^{1-\nicefrac{2}{p}}_{2,p,0}(\Dom))} + \|u_0\|^h_{L^{q}(\Omega \times D)}.
		\end{align}
		If $\theta \in (\nicefrac{1}{p}, \nicefrac{1}{2})$, then one has moreover the estimate
		\begin{align}
			\label{eq:ql_grow_bound3}
			\| u_R \|_{\L^p(\Omega; \Cont^{\theta-\nicefrac{1}{p}}([0,T]; \H^{1-2\theta}_0(\Dom)))} \lesssim_\theta 1 + \| u_0 \|_{\L^p(\Omega; \B^{1-\nicefrac{2}{p}}_{2,p,0}(\Dom))} + \|u_0\|^h_{L^{q}(\Omega \times D)}.
		\end{align}
		Then, the conclusion using the stochastic compactness argument works analogous, with one additional argument: first, by~\eqref{eq:moment_cutoff_nonlin}, the terms corresponding to $\Phi$ and $\phi$ in~\eqref{eq:martingale_candidates} satisfy the same moment condition as before. Second, for the almost sure convergence, decompose
		\begin{align}
			\label{eq:phi_decomp}
			\phi(\wt{u}) - {}_R\phi(\wt{u}_R) = (\phi(\wt{u})-\phi(\wt{u}_R)) + (\phi(\wt{u}_R) - {}_R\phi(\wt{u}_R)),
		\end{align}
		likewise for ${}_R\Phi_{i}$.
		As seen before, by Vitali's convergence theorem, it suffices to show $(r,x)$ almost everywhere convergence of $\lb {}_R\phi(\wt{u}_R), \xi \rb$ and $\lb {}_R\Phi_{i}(\wt{u}_R), \partial_i \xi \rb$, and upon passing to a subsequence, we can assume that $\wt{u}_R$ converges to $\wt{u}$ for almost every $(s,x)$. Therefore, the second term on the right-hand side of~\eqref{eq:phi_decomp} goes to zero since $\wt{u}_R$ is a bounded sequence for fixed $(s,x)$, so that eventually $\phi$ and ${}_R\phi$ coincide. For the first term on the right-hand side of~\eqref{eq:phi_decomp}, we upgrade almost everywhere convergence of the subsequence to $\L^{ph}((0,T) \times \Dom)$ convergence using Vitali's theorem (recall that $q>ph)$. Then the claim follows by continuity of the Nemytski operator~\cite[Thm.~3.2.24]{DM}, that is to say, continuity of the operator $v\mapsto \phi(t,x,v)$ from $\L^{ph}((0,T) \times \Dom) \to \L^p((0,T) \times \Dom)$. Here, we use that $\phi$ is a Carath\'eodory function satisfying the uniform growth condition stated in Assumption~\ref{ass:ql}~\ref{ass:2ndorder13}.
	\end{proof}

\bibliographystyle{plain}

\end{document}